\documentclass[12pt,reqno]{amsart}
\usepackage[headings]{fullpage}
\usepackage{amssymb,amsmath,mathtools,bbm,tikz,tikz-cd,color,
}
\usepackage{dsfont}
\usepackage{nicematrix}
\usepackage{fancybox}
\usetikzlibrary{positioning}
\usetikzlibrary{calc}
\usetikzlibrary{decorations.markings}
\usetikzlibrary{arrows.meta}
\usetikzlibrary{calc}
\usetikzlibrary{decorations.pathmorphing}
\usetikzlibrary{decorations.pathreplacing}

\tikzstyle over=[preaction={draw,line width=6pt,white}]
\tikzset{every path/.style={thick}
}
\tikzset{baseline={([yshift=-.5ex]current bounding box.center)} }
\tikzstyle{point}=[postaction={decorate,decoration={markings,
    mark=at position #1 with {\arrow{>}}}}]

\tikzstyle longleadsto=[
    {decorate, decoration={
    zigzag,
    segment length=4,
    amplitude=.9,
    post=lineto,
    post length=3pt,
}
    }
]

\allowdisplaybreaks

\usepackage[bookmarks=true,%
    colorlinks=true,%
    linkcolor=blue,%
    citecolor=blue,%
    filecolor=blue,%
    menucolor=blue,%
    urlcolor=blue,%
    breaklinks=true]{hyperref}
\usepackage{slashed}    
\usepackage{verbatim}
\usepackage[normalem]{ulem} 
\usepackage{diagbox}

\usetikzlibrary{fit} 

\newtheorem{theorem}{Theorem}[section]
\theoremstyle{definition}
\newtheorem{proposition}[theorem]{Proposition}
\newtheorem{lemma}[theorem]{Lemma}
\newtheorem{definition}[theorem]{Definition}
\newtheorem{remark}[theorem]{Remark}
\newtheorem{corollary}[theorem]{Corollary}
\newtheorem{conjecture}[theorem]{Conjecture}
\newtheorem{question}[theorem]{Question}

\def\red#1{\textcolor{red}{#1}} 
\def\blue#1{\textcolor{blue}{#1}}
\def\green#1{\textcolor{green}{#1}} 

\def\tr{\mathrm{tr}}

\def\BZ{\mathbbm Z}
\def\BQ{\mathbbm Q}

\def\BK{\mathbbm K}

\def\la{\langle}
\def\ra{\rangle}

\def\a{\alpha}
\def\b{\beta}
\def\g{\gamma}

\def\e{\epsilon}

\def\be{\begin{equation}}
\def\ee{\end{equation}}

\newcommand{\slt}{{\mathfrak{sl}(2)}}

\def\red#1{\textcolor[rgb]{1.00,0.00,0.00}{#1}}
\def\green#1{\textcolor[rgb]{0.00,0.70,0.30}{#1}}

\def\End{\mathrm{End}}

\def\LG{\mathrm{LG}}

\def\Alex{\Delta}
\def\diag{\mathrm{diag}}

\def\RT{F}
\def\mRT{\text{mRT}}

\def\tangle{T}

\newcommand{\ind}{\mathds{1}}
\def\pser#1{[\![#1]\!]}
\newcommand{\gl}{\ensuremath{\mathfrak{gl}}}

\makeatletter
\def\namedlabel#1#2{\begingroup
    #2%
    \def\@currentlabel{#2}%
    \phantomsection\label{#1}\endgroup
}
\makeatother

\newcommand{\ev}{\operatorname{ev}}
\newcommand{\coev}{\operatorname{coev}}
\def\rev{\widetilde\ev}
\def\rcoev{\widetilde\coev}

\newcommand{\slthree}{\mathfrak{sl}_3}

\def\homog{\mathrm{Homog}}
\def\normal{\mathcal{N}}

\makeatletter

\renewcommand\thepart{\@Roman\c@part}%
\renewcommand\part{%
   \if@noskipsec \leavevmode \fi
   \par
   \addvspace{6.7ex}%
   \@afterindentfalse
   \secdef\@part\@spart}
\def\@part[#1]#2{%
    \ifnum \c@secnumdepth >\m@ne
      \refstepcounter{part}%
      \addcontentsline{toc}{part}{Part~\thepart.\ #1}%
    \else
      \addcontentsline{toc}{part}{#1}%
    \fi
    {\parindent \z@ \raggedright
     \interlinepenalty \@M
     \normalfont
     \ifnum \c@secnumdepth >\m@ne
       \centering\large\scshape \partname~\thepart.%
       \hspace{1ex}%
     \fi%
     \large\scshape #2%
     \markboth{}{}\par}%
    \nobreak
    \vskip 4.7ex
    \@afterheading}
  \def\@spart#1{
  \refstepcounter{part}%
  \addcontentsline{toc}{part}{#1}%
    {\parindent \z@ \raggedright
     \interlinepenalty \@M
     \normalfont
     \centering\large\scshape #1\par}%
     \nobreak
     \vskip 4.7ex
     \@afterheading}
\renewcommand*\l@part[2]{%
  \ifnum \c@tocdepth >-2\relax
    \addpenalty\@secpenalty
    \addvspace{0.75em \@plus\p@}%
    \begingroup
      \parindent \z@ \rightskip \@pnumwidth
      \parfillskip -\@pnumwidth
      {\leavevmode
       \normalsize \bfseries #1\hfil \hb@xt@\@pnumwidth{\hss #2}}\par
       \nobreak
       \if@compatibility
         \global\@nobreaktrue
         \reverypar{\global\@nobreakfalse\reverypar{}}%
      \fi
    \endgroup
  \fi}

\def\l@subsection{\@tocline{2}{0pt}{2pc}{6pc}{}}
\makeatother

\begin{document}

\title[On some log-concavity properties of the Alexander and $\LG$ invariants]{On some log-concavity properties of the Alexander-Conway and Links-Gould invariants}

\author[M. Harper]{Matthew Harper}
\address{Michigan State University, East Lansing, Michigan, USA}
\email{mrhmath@proton.me}
   
\author[B.-M. Kohli]{Ben-Michael Kohli}
\address{Section de Math\'ematiques, Universit\'e de Gen\`eve \\
rue du Conseil-G\'en\'eral 7-9, 1205 Gen\`eve, Switzerland}
\email{bm.kohli@protonmail.ch}

\author[J. Song]{Jiebo Song}
\address{Beijing Institute of Mathematical Sciences and Applications,
  Huairou District, Beijing, China}
\email{songjiebo@bimsa.cn}

\author[G. Tahar]{Guillaume Tahar}
\address{Beijing Institute of Mathematical Sciences and Applications,
  Huairou District, Beijing, China}
\email{guillaume.tahar@bimsa.cn}


\date{\today }

\begin{abstract}
The Links--Gould invariant $\LG(L ; t_0, t_1)$ of a link $L$ is a two-variable quantum generalization of the Alexander--Conway polynomial $\Delta_L(t)$ and has been shown to share some of its most geometric features in several recent works. Here we suggest that $\LG$ likely shares another of the Alexander polynomial’s most distinctive - and mysterious - properties: for alternating links, the coefficients of the Links-Gould polynomial alternate and appear to form a log-concave two-indexed sequence with no internal zeros. The former was observed by Ishii for knots with up to 10 crossings. We further conjecture that they satisfy a bidimensional property of unimodality, thereby replicating a long-standing conjecture of Fox (1962) regarding the Alexander polynomial, and a subsequent refinement by Stoimenow. We also point out that the Stoimenow conjecture reflects a more structural phenomenon: after a suitable normalization, the Alexander polynomial of an alternating link appears to be a Lorentzian polynomial. We give {compelling} experimental and computational evidence for these different properties.
\end{abstract}

\maketitle

{\footnotesize
\tableofcontents
}


\section{Introduction}

\subsection{The Alexander and Links-Gould polynomials of links}

The Alexander polynomial $\Delta_L$ is an isotopy invariant of links $L$ valued in $\BZ[t^{\pm1}]$ and defined up to units. It was first defined in 1928 by James Waddell Alexander \cite{Alexander28} and was the first polynomial invariant of knots. Remarkably, the Alexander polynomial remained the only known polynomial link invariant until Vaughan Jones discovered his eponymous polynomial half a century later \cite{Jones85}. Designed to distinguish links up to isotopy, the Alexander polynomial also relates to some fundamental geometric properties of knots and links. For example, the degree of $\Delta_L(t)$ is a lower bound for the 3-genus of $L$ \cite{Seifert35}, and the Alexander polynomial of a fibered knot is monic \cite{Neuwirth65, Rapaport60, Stallings62}. The Alexander invariant is also relevant to study the four-dimensional properties of a knot. For example, for a topologically slice knot $K$, $\Delta_K(t)$ takes a particular form known as the Fox-Milnor condition \cite{FM66}.

It is also worth noting that $\Delta_L(t)$ can be constructed as a quantum link invariant, that is, a Reshetikhin--Turaev-type invariant where it occurs in its palindromic normalization as the Alexander--Conway polynomial. It arises in this way by considering representations of either: the Hopf algebra $U_q\slt$ at fourth roots of unity, or the Hopf superalgebra $U_q\gl(1|1)$, see \cite{murakami2006multi, Viro07}. In the latter perspective, the Alexander polynomial is the first in the family of Links--Gould invariants $\LG^{m,n}(L;t_{0},t_{1}) \in \mathbb{Z}[t_0^{\pm 1} , t_1^{\pm 1}]$ obtained from the Reshetikhin--Turaev construction applied to representations of the Hopf superalgebras $U_{q}\mathfrak{gl}(m \vert n)$, see \cite{Links-Gould, DWKL, DW01}. The simplest case after the Alexander polynomial is derived from $U_{q}\mathfrak{gl}(2 \vert 1)$. We shall denote the associated invariant $\LG^{2,1}$ by $\LG$ and refer to it as the Links--Gould polynomial. 

Although $\LG$ is not yet known to have a classical geometric interpretation that would relate it even more closely to the Alexander polynomial, recent results show that it reproduces some of the Alexander polynomial's most distinctive characteristics. Specifically, the degree of the Links--Gould polynomial gives a lower bound for the genus of a knot by work of two of the authors \cite{Kohli-Tahar}. Also, the $\LG$ invariant of a fibered link is $\mathbb{Z}[q^{\pm 1}]$-monic following the recent work of L\'opez-Neumann and van der Veen \cite{LNVdV25}. These two properties imply and improve upon similar statements for the Alexander polynomial. This is because $\LG(L;t_0,t_1)$ specializes to $\Delta_L(t)$ in at least two different ways \cite{DWIL, Kohli, KPM}: 
\be
\label{LGspec}
\LG(L;t_0,-t_0^{-1}) = \Alex_L(t_0^2)\,, \qquad
\LG(L; t_0, t_0^{-1}) = \Alex_L(t_0)^2.
\ee
It is therefore reasonable to imagine that other properties of the Alexander polynomial should have their counterparts for the Links-Gould invariant.

\subsection{The Alexander polynomial of alternating links} A class where the Alexander polynomial has an interesting behavior is that of alternating links: the Alexander polynomial of an alternating link is alternating with no internal zeros by works of Crowell \cite{Cro} and Murasugi \cite{Mur}. This means that if we consider the Conway normalization of $\Delta_L(t)$, we can write:
$$
\Delta_L(t) = \sum_{k=-n}^{n} a_k t^k, \quad \text{with } a_l \cdot a_{l+1} < 0 \text{ for all } l \in \pser{-(n-1), n-1}\,.
$$
Moreover, Fox conjectured that the absolute values of the coefficients of the Alexander polynomial of an alternating link form a unimodal sequence \cite{Fox}, i.e. 
$$
|a_{-n}| \leq |a_{-n+1}| \leq \cdots \leq |a_k| = |a_{k+1}| = \cdots = |a_m| \geq |a_{m+1}| \geq \cdots \geq |a_n| \text{, } -n \leq k \leq m \leq n \,.
$$
Stoimenow refined the statement by conjecturing that the sequence should be log-concave with no internal zeros \cite{Sto}. This in turn would imply that the sequence is trapezoidal:
$$
|a_{-n}| < |a_{-n+1}| < \cdots < |a_k| = |a_{k+1}| = \cdots = |a_m| > |a_{m+1}| > \cdots > |a_n| \text{, } -n \leq k \leq m \leq n \,.
$$
Different forms of the Fox conjecture have been proved over the years for many infinite families of knots and links, see \cite{Par, Hartley, Mur85, OZ, Jong, Sto16, AC, Hafner, azarpendar24, kalman25}. A summary of these studies and their scope can be found in \cite{azarpendar24}. It also includes several recent works \cite{AC, Hafner, azarpendar24, kalman25}. The general conjecture, however, remains unsolved to this day.

\subsection{Similar properties for the Links-Gould polynomial of alternating links?} In his early foundational studies on the Links-Gould invariant, Ishii conjectured that $\LG$ displays the same ``alternating'' behavior as the Alexander polynomial on alternating knots.

\begin{conjecture}[\cite{Ishii}]\label{conj.0}
    The LG polynomial $\LG(K ; t_0, t_1) = \sum_{i,j} a_{ij} t_0^i t_1^j$ of an alternating knot $K$ is ``alternating'': 
$a_{ij} a_{i'j'} \geq 0 \quad \text{if } i + j - i' - j' \text{ is even, and} \quad
a_{ij} a_{i'j'} \leq 0 \quad \text{otherwise}$.
\end{conjecture}
Ishii verified Conjecture~\ref{conj.0} for all prime knots with up to $10$ crossings. Using the fact that $\LG$ is equivalent by a change of variables to the Garoufalidis--Kashaev $V_1$-polynomial of links \cite{GK:multi,GHKST,GHKKST}, the computations of Garoufalidis--Li for $V_1$ show that the Ishii conjecture is true for all prime knots with up to $16$ crossings, see \cite{GL:patterns}. \\

In this work we conjecture that once again the Links-Gould invariant $\LG$ and the Alexander polynomial follow similar patterns. First, Conjecture~\ref{conj.0} can be made more precise, echoing the behavior of the Alexander polynomial.

\begin{definition}
    A two-indexed sequence $(a_{ij})_{i,j \in \mathbb{Z}}$ is said to have \emph{no interior zeros} if 
    $$
    \text{Conv}(\text{Supp}(a_{ij})) \cap \mathbb{Z}^2 = \text{Supp}(a_{ij}) \,,
    $$
    where $
    \text{Supp}(a_{ij}) = \{ (i,j) \in \mathbb{Z}^2 \, : \, a_{ij} \neq 0   \} 
    $ and $\text{Conv}(X)$ is the convex hull of a subset $X$ of $\mathbb{R}^2$.
\end{definition}

\begin{conjecture}[strong version of the Ishii conjecture]\label{conj.1}
    The LG polynomial $\LG(L ; t_0, t_1) = \sum_{i,j} a_{ij} t_0^i t_1^j$ of an \emph{alternating link} $L$ with $\mu$ components is ``alternating'' with 
    $$a_{ij} = (-1)^{\mu + i + j + 1} \, |a_{ij}|  $$
    for all $i,j$. Moreover, $(a_{ij})$ has no interior zeros.
\end{conjecture}

\begin{remark}
    Conjecture~\ref{conj.1} is \emph{not} true for other two-variable generalizations of the Alexander polynomial, i.e. for the two variable polynomial invariant $\Delta_{\slthree}$ derived from representations of $U_q\mathfrak{sl}_3$ at a fourth root of unity, studied by one of the authors \cite[Figure 10]{Harper2020}.
\end{remark}

We further propose that $\LG$ should satisfy bidimensional properties of log-concavity and unimodality, thereby generalizing, at least in spirit, the Fox and Stoimenow conjectures.

\begin{definition}
A sequence $(a_{i})_{i \in \mathbb{Z}}$ is \emph{log-concave} if, for every \( i \in \mathbb{Z} \), we have: $a_{i-1} \, a_{i+1} \leq a_i^2$.
\end{definition}

\begin{remark}
    Log-concavity, as well as its generalizations we will define hereafter, is invariant under multiplication by a monomial.
\end{remark}

\begin{definition}
   We will refer to a bidimensional sequence $(a_{ij})_{i,j \in \mathbb{Z}}$ as \emph{2d log-concave} if for every $i,j,k,l \in \mathbb{Z}$ :
   $$a_{i+k,j+l} \, a_{i-k,j-l} \leq a_{ij}^2\,.$$
\end{definition}

We understand 2d log-concavity as being indicative of the sequence being log-concave ``in all possible directions''. In particular if a sequence is 2d log-concave, its restriction to any row, column or diagonal in $\mathbb{Z}^2$ is a log-concave singly-indexed sequence.

\begin{conjecture}[2d log-concavity]\label{conj.2}
For any alternating link $L$, the absolute values of the coefficients of $\LG(L ; t_0, t_1)$ form a 2d log-concave sequence (with no interior zeros).
\end{conjecture}

\begin{definition}
    We say that a two-indexed sequence $(a_{ij})_{i,j \in \mathbb{Z}}$ is \emph{unimodal} if, for any $K \geq 0$, the two-indexed sequence $(a_{ij} \, \ind_{|a_{ij}| \geq K})_{i,j \in \mathbb{Z}}$ has no interior zeros.
\end{definition}

\begin{conjecture}[unimodality]\label{conj.3}
For any alternating link $L$, the 2d sequence of absolute values of the coefficients of $\LG(L ; t_0, t_1)$ is unimodal. 
\end{conjecture}

We will give evidence in support of these assumptions. Specifically, we show that Conjectures~\ref{conj.1}, \ref{conj.2}, and \ref{conj.3} hold for all alternating prime knots with up to 12 crossings, as well as for several infinite families of alternating links. 

Note that despite the obvious similarities between Conjectures~\ref{conj.1}, \ref{conj.2}, and \ref{conj.3} and comparable statements for the Alexander polynomial, these conjectures are not straightforward generalizations of the well-known properties for $\Delta_L(t)$. None of the specializations from Equation~\eqref{LGspec} allow the properties of the Alexander polynomial on alternating links to be understood as consequences of the properties of the Links--Gould polynomial. This was pointed out in \cite{Kohli2} about the Ishii conjecture, and remains true here. This suggests that unlike the genus bounds and fiberedness criteria, these properties for alternating links are more likely the trace of a deep structural behavior common to both link invariants.

Let us point out that in our journey to understand whether the $\LG$ polynomial is more closely related to the Jones polynomial or to the Alexander polynomial, these new hypothetical properties of the Links-Gould polynomial reinforce its proximity with $\Delta_L(t)$. Indeed, the Jones polynomial of an alternating link is alternating \cite{Thi}, however, it is not always unimodal. Consider, for example,
$$
V_{8_{10}}(t) = -t^{-6} +2 t^{-5} -4 t^{-4} + 5 t^{-3} -4 t^{-2} +5 t^{-1} -3 + 2t -t^2.
$$

\subsection{A strengthening of the Fox-Stoimenow conjecture for the Alexander polynomial}

In recent work, Hafner, Mészáros, and Vidinas proved the Fox--Stoimenow conjecture for special alternating links \cite{Hafner} by showing that a well chosen normalization and homogenization of the Alexander polynomial is Lorentzian, a property introduced and studied by Br{\"a}nd{\'e}n and Huh \cite{branden2020lorentzian}. Observe that being Lorentzian is a very strong assumption and, as we will see, it is not satisfied by the $\LG$ invariant for any obvious choice of variables, not even for all \emph{special} alternating links. However, in the case of $\Delta_L(t)$, we show that for \emph{any} alternating link, the Stoimenow extension of the Fox conjecture is equivalent to the {divided power normalization of the} Alexander polynomial being Lorentzian. We refer to the Definitions \ref{defn:Lorentzian} and \ref{defn:Normalization} for precise definitions of the (divided power) normalization and Lorentzian property.

\begin{proposition}[Lorentzian property for Alexander] \label{conj.4}
The homogenized Alexander polynomial of an alternating link $\mathrm{Homog}(\Delta_L(-t))$ is denormalized Lorentzian if and only if the absolute values of the coefficients of the Alexander polynomial of the link form a log-concave sequence (with no internal zeros).
\end{proposition}

That the Alexander polynomial is denormalized Lorenztian follows from the structure of $\Delta_L(t)$ for $L$ an alternating link and from Proposition \ref{conj.4}, which is an application of the following elementary algebraic result.

\begin{proposition}\label{algebraic lemma}
For $P(t)$ a \emph{one-variable} polynomial with nonnegative real coefficients, its sequence of coefficients is log-concave if and only if $Q(t,z) := \mathrm{Homog}(P)$ is denormalized Lorentzian.
\end{proposition}

Thus, following Stoimenow \cite{Sto} the Alexander polynomial satisfies the Lorentzian property for all alternating knots with up to 16 crossings, and this is conjecturally true for all alternating links. \\

We conclude this introduction with a question.

\begin{question}
    Is the homogenized Links--Gould polynomial of alternating links denormalized Lorentzian after a proper change of variables?
\end{question}


\subsection{Organization of the paper} Section~\ref{sec:DefLG} contains the definition and elementary properties of the Links--Gould link polynomial. Some basic properties of log-concave sequences are presented in Section~\ref{sec:logconcave}, as well as the definition and characteristics of Lorentzian polynomials. In particular, log-concavity is related to unimodality and being a trapezoidal sequence. These properties will be useful to simplify the calculations carried out later in the paper. Propositions~\ref{conj.4} and \ref{algebraic lemma} are proved in Section~\ref{sec:Proof} and the implications for the Fox--Stoimenow conjecture are discussed. Section~\ref{sec:Evidence} presents experimental and computational evidence in support of Conjectures~\ref{conj.1}, \ref{conj.2} and \ref{conj.3}. Indeed, using a \texttt{Python} program \cite{data} that analyzes the data from the $\LG$-explorer \cite{LGexplorer}, we confirm that for all alternating knots available in the $\LG$-explorer database, these conjectures are true. This includes all alternating knots with up to 12 crossings, and a selection of 13 to 16 crossing alternating knots with small braid index. We also show that the Links--Gould polynomial of the trefoil is not denormalized Lorenztian for various choices of variables.  Moreover, we prove that all three conjectures are true for two infinite families of links: \texorpdfstring{$(2,n)$}{(2,n)}-torus links and twist knots. Their Links--Gould invariants are presented in the Figures of Section \ref{sec.values}.

\subsection*{Acknowledgements}
The authors express their gratitude to David de Wit for his work implementing the $\LG$-explorer, and to Jon Links for making it accessible on his webpage. MH was partially supported through the NSF-RTG grant
\#DMS-2135960. BMK was partially supported through the BJNSF grant IS24066. GT was supported by the BJNSF grant IS23005.

\section{The Links-Gould polynomial: definition and properties}\label{sec:DefLG}

In this section we closely follow the ideas exposed in \cite{GHKST, GHKKST}. For a more representation theoretic approach to the Links--Gould invariant see \cite{DWKL,GHKW}. 

\subsection{The Reshetikhin--Turaev functor}
\label{sub:RT}

We review briefly the well-known Reshetikhin--Turaev functor
~\cite{RT:ribbon,Tu:book} from tangles to tensor products of endomorphisms of a vector
space and its dual. 
  
A tangle diagram can be decomposed into a finite number of elementary
oriented tangle diagrams shown in Figure~\ref{fig:tangles1}. Moreover, up to
  isotopy, we may assume that for any horizontal line drawn on the diagram there is
  at most one critical point or a single crossing.

  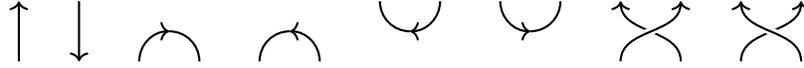
\begin{figure}[htpb!]
\begin{tikzpicture}[xscale=0.4, yscale =0.4]
\draw[->] (0,0) to (0,2);
\draw[->] (2,2) to (2,0);
\draw[->] (4,0) to [out=90, in=180] (5,1);
\draw[] (5,1) to [out=0, in=90] (6,0);
\draw[->] (10,0) to [out=90, in=0] (9,1);
\draw[] (9,1) to [out=180, in=90] (8,0);
\draw[->] (14,2) to [out=-90, in=0] (13,1);
\draw[] (13,1) to [out=180, in=-90] (12,2);
\draw[->] (16,2) to [out=-90, in=180] (17,1);
\draw[] (17,1) to [out=0, in=-90] (18,2);
\draw[->] (22,0) to [out=90, in=-90] (20,2);
\draw[line width=3pt, white] (20,0) to [out=90, in=-90] (22,2);
\draw[->] (20,0) to [out=90, in=-90] (22,2);
\draw[->] (24,0) to [out=90, in=-90] (26,2);
\draw[line width=3pt, white] (26,0) to [out=90, in=-90] (24,2);
\draw[->] (26,0) to [out=90, in=-90] (24,2);
\end{tikzpicture}
\caption{Elementary oriented tangles.}
\label{fig:tangles1}
\end{figure}

Set $V$ a vector space over a field $\BK$ of characteristic zero and consider
$V^\ast = \End(V,\BK)$ its dual. Each elementary oriented tangle is associated
to a linear map under the Reshetikhin-Turaev functor
$\RT$, as shown in Figure~\ref{fig:FIG3}. By composing these elementary maps, one obtains a linear map $\RT_\tangle$
for any oriented tangle diagram $\tangle$.
For example, if $\tangle$ is an oriented $(n,m)$-tangle diagram, then
$
\RT_{\tangle} \in \End(V^{\epsilon_1} \otimes V^{\epsilon_2} \otimes \ldots
\otimes V^{\epsilon_n}, V^{\delta_1} \otimes V^{\delta_2} \otimes \ldots
\otimes V^{\delta_m}) \,,
$
where $\epsilon_i$, $\delta_j$ are empty or $\ast$, depending on the
orientations of the boundary $\partial T$ of $T$.

\begin{figure}[htpb!]
\tikzset{every path/.style={thin}
}
\begin{tikzpicture}[xscale=.4, yscale=.4] 
\draw[thick,->] (0,0) to (0,2);
\node at (1,0) {$V$};
\node at (1,2) {$V$};
\draw[->] (1,.5) to (1,1.5);
\node[right] at (1,1) {$\mathrm{id}_V$};
\end{tikzpicture}   
\qquad
\begin{tikzpicture}[xscale=.4, yscale=.4]
\draw[thick,->] (7,0) to [out=90, in=-90] (5,2);
\draw[line width=3pt, white] (5,0) to [out=90, in=-90] (7,2);
\draw[thick,->] (5,0) to [out=90, in=-90] (7,2);
\node[below] at (9.5,.5) {$\phantom{W}V\otimes W\phantom{V}$};
\node[above] at (9.5,1.5) {$\phantom{V}W\otimes V\phantom{W}$};
\draw[->] (9.5,.5) to (9.5,1.5);
\node[right] at (9.5,1) {$R_{VW}$};
\end{tikzpicture}   
\qquad
\begin{tikzpicture}[xscale=.4, yscale=.4]
\draw[thick,->] (4,0) to [out=90, in=180] (5,1.5);
\draw[thick,] (5,1.5) to [out=0, in=90] (6,0);
\node[below] at (8,.5) {$\phantom{^*}V\otimes V^*$};
\node[above] at (8,1.5) {$\mathbb{K}$};
\draw[->] (8,.5) to (8,1.5);
\node[right] at (8,1) {$\rev_V$};
\end{tikzpicture}
\qquad 
\begin{tikzpicture}[xscale=.4, yscale=.4]
\draw[thick,->] (10,0) to [out=90, in=0] (9,1.5);
\draw[thick,] (9,1.5) to [out=180, in=90] (8,0);
\node[below] at (12.5,.5) {$V^*\otimes V\phantom{^*}$};
\node[above] at (12.5,1.5) {$\mathbb{K}$};
\draw[->] (12.5,.5) to (12.5,1.5);
\node[right] at (12.5,1) {$\ev_V$};
\end{tikzpicture}
\\[2em]
\begin{tikzpicture}[xscale=.4, yscale=.4]
\draw[thick,->] (0,2) to (0,0);
\node at (1.25,0) {$\phantom{^*}V^*$};
\node at (1.25,2) {$\phantom{^*}V^*$};
\draw[->] (1.25,.5) to (1.25,1.5);
\node[right] at (1.25,1) {$\mathrm{id}_{V^*}$};
\end{tikzpicture}   
\qquad
\begin{tikzpicture}[xscale=.4, yscale=.4]
\draw[thick,->] (5,0) to [out=90, in=-90] (7,2);
\draw[line width=3pt, white] (7,0) to [out=90, in=-90] (5,2);
\draw[thick,->] (7,0) to [out=90, in=-90] (5,2);
\node[below] at (9.5,.5) {$\phantom{V}W\otimes V\phantom{W}$};
\node[above] at (9.5,1.5) {$\phantom{W}V\otimes W\phantom{V}$};
\draw[->] (9.5,.5) to (9.5,1.5);
\node[right] at (9.5,1) {$R_{VW}^{-1}$};
\end{tikzpicture}   
\qquad
\begin{tikzpicture}[xscale=.4, yscale=.4]
\draw[thick,->] (6,2) to [out=-90, in=0] (5,.5);
\draw[thick,] (5,.5) to [out=180, in=-90] (4,2);
\node[above] at (8,1.5) {$\phantom{^*}V\otimes V^*$};
\node[below] at (8,.5) {$\mathbb{K}$};
\draw[->] (8,.5) to (8,1.5);
\node[right] at (8,1) {$\coev_V$};
\end{tikzpicture}
\qquad 
\begin{tikzpicture}[xscale=.4, yscale=.4]
\draw[thick,->] (4,2) to [out=-90, in=180] (5,.5);
\draw[thick,] (5,.5) to [out=0, in=-90] (6,2);
\node[above] at (8,1.5) {$V^*\otimes V\phantom{^*}$};
\node[below] at (8,.5) {$\mathbb{K}$};
\draw[->] (8,.5) to (8,1.5);
\node[right] at (8,1) {$\rcoev_V$};
\end{tikzpicture}
\caption{
  A graphical definition of the Reshetikhin-Turaev functor on oriented
  elementary tangle diagrams.}
\label{fig:FIG3}
\end{figure}
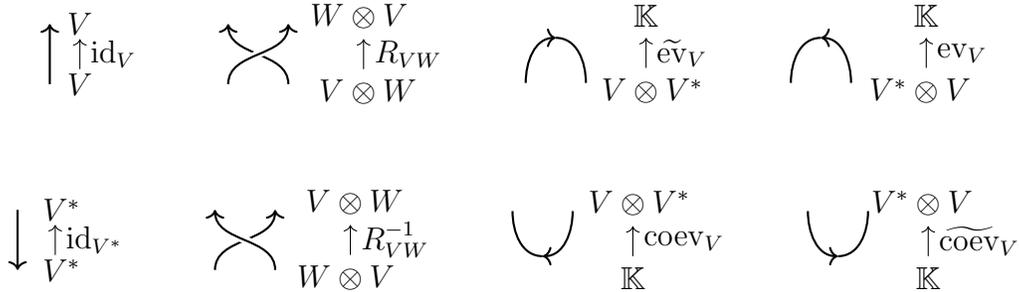

\subsection{Invariants of tangles}
\label{sub:tangles}

To state the identities of an enhanced $R$-matrix we need to fix some notation.
We do so following Ohtsuki~\cite{Ohtsuki}.
Fix a finite dimensional vector space $V$ with basis $(e_1, \ldots, e_n)$.
For $A \in \End(V \otimes V)$, we write $A = (A_{(i,j)}^{(k,l)})$ as a matrix in
basis $(e_i \otimes e_j)$. Therefore
$A_{(i,j)}^{(k,l)} = (e_k \otimes e_l)^{\ast}(A(e_i \otimes e_j))$. Then matrices $A^{\circlearrowleft}$ and $A^{\circlearrowright}$ are defined by setting: 
\[(A^\circlearrowleft)_{(k,i)}^{(l,j)} =A_{(i,j)}^{(k,l)}, \qquad
(A^\circlearrowright)_{(j,l)}^{(i,k)}=A_{(i,j)}^{(k,l)} \,.\]

\begin{definition}
\label{def.Rh}
An \emph{enhanced $R$-matrix} on a finite dimensional vector space $V$
is a pair of invertible endomorphisms $R \in \End(V\otimes V)$,
the \emph{$R$-matrix}, and $h \in \End(V)$, the \emph{enhancement}, satisfying
\begin{subequations}
\begin{align}
\label{eq1a}
R \circ (h \otimes h) &= (h \otimes h) \circ R \\
\label{eq1b}
\tr_{2} ((\mathrm{id}_V \otimes h) \circ R^{\pm 1}) &= \mathrm{id}_V \\
\label{eq1c}
(R^{-1})^{\circlearrowleft} \circ ((\mathrm{id}_V \otimes h)
\circ R \circ (h^{-1} \otimes \mathrm{id}_V))^{\circlearrowright} &=
\mathrm{id}_V \otimes \mathrm{id}_{V^{\ast}} \\
\label{eq1d}
(R \otimes \mathrm{id}_V) \circ (\mathrm{id}_V \otimes R)
\circ (R \otimes \mathrm{id}_V) &= (\mathrm{id}_V \otimes R)
\circ (R \otimes \mathrm{id}_V) \circ (\mathrm{id}_V \otimes R)
\end{align}
\end{subequations}
\end{definition}

The partial trace $\tr_2$ of $f \otimes g \in \End(V \otimes V)$
is defined by
\[
\tr_{2}(f \otimes g) = \tr(g) f \in \End(V)
\]
with the natural identification of $\End(V \otimes V)$ and
$\End(V) \otimes \End(V)$. {This extends to partial trace operations
$\tr_i:\End(V^{\otimes n})\to \End(V^{\otimes n-1})$ where we trace in the $i$-th
component. Let $\tr_{i_1i_2\cdots i_j}$ denote the composition of partial traces
$\tr_{i_1}\circ\tr_{i_2}\circ\cdots\circ\tr_{i_j}$.}

For $1\leq i <j \leq n$ and $f\in\End(V\otimes V)$ define
$(f)_{ij}\in\End(V^{\otimes n})$ which acts by $f$ in the $i$-th and $j$-th
tensor factors and is the identity otherwise.

Given an $R$-matrix, Equation \eqref{eq1d} defines a representation $\rho_R$ of
the braid group $B_n$ by mapping the elementary braid
generator in position $(k, k+1)$ to $(R)_{k,k+1}\in \End(V^{\otimes n})$.

\begin{remark}
\label{rem.cups}
An enhanced $R$-matrix determines an operator valued invariant of isotopy classes
of tangles under the Reshetikhin-Turaev functor if one considers the following definitions of
cup and cap maps:
\begin{subequations}
\begin{align*}
\rev_V(x \otimes f) & := f(h(x)) \\
\ev_V(f \otimes x) & := f(x) \\
\coev_V(1) & := \displaystyle\sum_{i} e_i^\ast \otimes e_i \\
\rcoev_V(1) & := \displaystyle\sum_{i} h^{-1}(e_i) \otimes e_i^\ast \,.
\end{align*}
\end{subequations}
\end{remark}

\subsection{Invariants of links}

Assuming this, an enhanced $R$-matrix on $V$ gives rise to an invariant $\RT$ of
closed links. However, for some enhanced $R$-matrices (e.g. for the $\LG$ invariant) $\RT$ is the zero invariant.
To obtain a nontrivial invariant, we consider tangles,
where an enhanced $R$-matrix gives an $\End(V)$-valued invariant $\RT_\tangle$ of a
$(1,1)$-tangle $\tangle$ and an $\End(V\otimes V)$-valued invariant $\RT_\tangle$
of a  $(2,2)$-tangle with upward oriented boundary, by which we mean a tangle
with two upward incoming and outgoing arcs and perhaps additional closed components.
To get to a scalar-valued invariant of links, we need to ensure: 

\begin{enumerate}
\item[(\namedlabel{item:P1}{$P_1$})]
  For every $(1,1)$-tangle $\tangle$, $\RT_\tangle$ is a scalar multiple of
  $\mathrm{id}_V$. 
\item[(\namedlabel{item:P2}{$P_2$})]
  For every $(2,2)$-tangle $\tangle$ defining a map $F_T\in\End(V^{\otimes 2})$
  with left and right closures $\tangle_1$ and $\tangle_2$, we have
  $\RT_{\tangle_1}=\RT_{\tangle_2} \in \End(V)$.
\end{enumerate}

\begin{figure}[htpb!]
\begin{tikzpicture}[baseline=10, xscale=1, yscale=1]
\node[rectangle,inner sep=3pt,draw] (c) at (1/2,1/2) {$\tangle$};
\draw[thick,->] ($(c.north east)+(-.1,0)$) to [out=90, in= 180]
($(c.north east)+(.4,1/2)$) to [out=0, in= 90] ($(c.east)+(.9,0)$);
\draw[thick] ($(c.east)+(.9,0)$) to [out=-90, in= 0] ($(c.south east)+(.4,-1/2)$)
to [out=180, in= -90] ($(c.south east)+(-.1,0)$);
\draw[thick] ($(c.south west)+(.1,-1/2)$) to ($(c.south west)+(.1,0)$);
\draw[thick,->] ($(c.north west)+(.1,0)$) to ($(c.north west)+(.1,1/2)$);
\node[] at ($(c)+(.25,-1.5)$) {$\tangle_1$};
\end{tikzpicture}
\qquad
\begin{tikzpicture}[baseline=10, xscale=1, yscale=1]
\node[rectangle,inner sep=3pt,draw] (c) at (1/2,1/2) {$\tangle$};
\draw[thick,->] ($(c.north west)+(.1,0)$) to [out=90, in= 0]
($(c.north west)+(-.4,1/2)$) to [out=180, in= 90] ($(c.west)+(-.9,0)$);
\draw[thick] ($(c.west)+(-.9,0)$) to [out=-90, in= 180] ($(-.4,-1/2)+(c.south west)$)
to [out=0, in= -90] ($(c.south west)+(.1,0)$);
\draw[thick] ($(c.south east)+(-.1,-1/2)$) to ($(c.south east)+(-.1,0)$);
\draw[thick,->] ($(c.north east)+(-.1,0)$) to ($(c.north east)+(-.1,1/2)$);
\node[] at ($(c)+(-.25,-1.5)$) {$\tangle_2$};
\end{tikzpicture}   
\caption{The left and right closures $\tangle_1$ and $\tangle_2$ of a $(2,2)$-tangle
  $\tangle$.}
\label{fig:T12}
\end{figure}
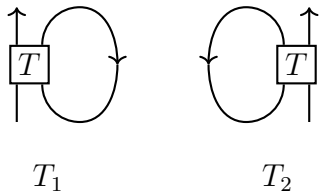

When \eqref{item:P1} is satisfied and $T$ is a $(1,1)$-tangle, we write
$F_T=\la F_T\ra \mathrm{id}_V$. Properties \eqref{item:P1} and \eqref{item:P2}
imply the existence of a scalar invariant of links $L$ obtained by cutting a link
along any component to obtain a $(1,1)$-tangle $L^\mathrm{cut}$. 
We denote this modified Reshetikhin-Turaev link invariant 
$\mRT_L=\la F_{L^\mathrm{cut}}\ra$.

\begin{remark}
When $\RT$ is determined by a ribbon category, Property \eqref{item:P1} is the
assumption that the coloring object $V$ is simple. Property \eqref{item:P2} is the
assumption that $V$ is an ambidextrous object in the sense of \cite{GPT}.
\end{remark} 

\begin{remark}
\label{rem.Lcut}
If $\beta\in B_n$ is a braid with closure $L$, then 
\begin{equation*}
    \begin{aligned}
\RT_{L^{\mathrm{cut}}}&=
\tr_{2,\dots, n}\left((\mathrm{id}_V\otimes
h^{\otimes (n-1)})\circ \rho_R(\b) \right)\in \End(V)
\\
\la\RT_{L^{\mathrm{cut}}}\ra
&=
\frac{1}{\dim(V)}
\tr\left((\mathrm{id}_V\otimes
h^{\otimes (n-1)})\circ \rho_R(\b)\right)\,.
\end{aligned}
\end{equation*}
\end{remark}

\subsection{The \texorpdfstring{$\LG$}{LG} invariant of links}
Note that in dealing with the Links--Gould invariant of links, we stick to
conventions used by Ishii for example in \cite{Ishii}. In doing so,
$\LG(L;p^{-2}, p^{2} q^{2})$ with $p = q^{\alpha}$ coincides with the Links-Gould
invariant from \cite{DWKL}.

Suppose $V$ is a vector space over a field $\BK$ with an ordered basis
$(v_1,\dots,v_n)$ and an $R$ matrix $R \in \End(V \otimes V)$.
Abbreviating $v_{ij}=v_i \otimes v_j$ for $i,j=1,\dots n$, the $n^2 \times n^2$
matrix $R$ can be presented
as an $n \times n$ matrix $\mathsf{R}:=(R(x_{ij}))_{1 \leq i,j \leq n}$ whose
entries are $\BK$-linear combinations of the $v_{ij}$.

The Links--Gould $R$-matrix is defined as follows.
Consider a 4-dimensional $\BQ(t_0,t_1)$-vector space $W$ with an ordered basis
$(w_1,w_2,w_3,w_4)$. Abbreviating $w_{ij}=w_i \otimes w_j$ for $i,j=1,\dots 4$, the $R$-matrix
$\mathsf{R}_{\LG}:=((R_{\LG})(w_{ij}))_{1 \leq i,j \leq 4}$ is given by
\begin{center}
\small
\resizebox{\textwidth}{!}{
$
  \mathsf{R}_{\LG}=\left(\begin{array}{@{}cccc}
  t_0 w_{11}
  &
  t_0^{1/2} w_{21}
  &
  t_0^{1/2} w_{31}
  &
   w_{41}
  \\ 
  t_0^{1/2} w_{12} + (t_0-1)w_{21}
  &
  - w_{22}
  &
  (t_0t_1-1) w_{23}
  -t_0^{1/2}t_1^{1/2} w_{32}
  -t_0^{1/2}t_1^{1/2}Y w_{41}
  &
  t_1^{1/2} w_{42}
  \\ 
  t_0^{1/2} w_{13} + (t_0-1)w_{31}
  &
  -t_0^{1/2}t_1^{1/2} w_{23}+Y w_{41}
  &
  - w_{33}
  &
  t_1^{1/2} w_{43}
  \\ 
  w_{14}-t_0^{1/2}t_1^{1/2}Y w_{23}+Y w_{32}+Y^2 w_{41}
  &
  t_1^{1/2} w_{24}+(t_1-1) w_{42}
  &
  t_1^{1/2} w_{34}+(t_1-1)w_{43}
  &
  t_1 w_{44}
  \end{array}\right)
$}
\end{center} 
with $Y = \sqrt{(t_0-1)(1-t_1)}$. 

\begin{remark}
    The entries in the above matrix are
in the quadratic extension $\BQ(t_0,t_1)[Y]$ of the field $\BQ(t_0,t_1)$. 
\end{remark}

Define also $h_{\LG}  =\diag(t_0^{-1}, -t_1, -t_0^{-1}, t_1) \in \End(W)$.

\begin{lemma}
$(R_{\LG},h_{\LG})$ is an enhanced
$R$-matrix and satisfies properties \eqref{item:P1} and \eqref{item:P2}.
\end{lemma}

\begin{definition}
    The Links--Gould polynomial invariant of a link $L$ is defined from the previous enhanced $R$-matrix by setting 
    $
    \LG(L;t_0,t_1) = \la\RT_{L^{\mathrm{cut}}}\ra.
    $
\end{definition}

\begin{remark}
    Though it is not obvious at first glance, one can show that $\LG(L;t_0,t_1) \in \mathbb{Z}[t_0^{\pm 1} , t_1^{\pm 1}]$, see \cite[Theorem 1]{Ishii} and \cite[Theorem 1]{GHKKST}.
\end{remark}

The elementary properties of $\LG$ mirror features of the Alexander-Conway polynomial, some of which are listed below.

\begin{proposition}
The Links-Gould polynomial satisfies the following properties.
\begin{enumerate}
    \item $\LG(\bigcirc) = 1$\,.
    \item Denoting $L^{*}$ the reflection of $L$, we have
    \[
        \LG(L^{*}; t_{0}, t_{1}) = \LG(L; t_{0}^{-1}, t_{1}^{-1})\,.
    \]
    \item The symmetry
    \[
        \LG(L; t_{0}, t_{1}) = \LG(L; t_{1}, t_{0})
    \]
    holds. When taken together with the fact that link inversion switches the variables $t_0$ and $t_1$, it implies that $\LG$ does not detect inversion.
    \item For two links $L$ and $L'$, denoting their connected sum by $L \# L'$, we have
    \[
        \LG(L \# L') = \LG(L)\, \LG(L')\,.
    \]
    \item If $L = L' \sqcup L''$ is the split union of links $L'$ and $L''$, then
    \[
        \LG(L) = 0\,.
    \]
\end{enumerate}
\end{proposition}

A (rather complicated) complete set of skein relations for $\LG$ was recently presented in \cite{GHKKST}.

\section{Log-concavity, unimodality, Lorentzian polynomials: properties, examples, counter-examples}\label{sec:logconcave}

\subsection{Log-concavity, unimodality}

First we give some elementary properties of log-concave sequences that will be helpful in the next section. The following definitions, as mentioned in the introduction, are recalled below. \\

We say that a sequence $(a_{i})_{i \in \mathbb{Z}}$ is \emph{log-concave} if, for every \( i \in \mathbb{Z} \), we have: $$a_{i-1} \, a_{i+1} \leq a_i^2\,.$$
Likewise, we refer to a bidimensional sequence $(a_{ij})_{i,j \in \mathbb{Z}}$ as \emph{2d log-concave} if for every $i,j,k,l \in \mathbb{Z}$ :
   $$a_{i+k,j+l} \, a_{i-k,j-l} \leq a_{ij}^2\,.$$

\begin{proposition}
    For $(a_{i})_{i \in \mathbb{Z}}$ a log-concave sequence of non-negative real numbers with no internal zeros, the sequence $(a_i/a_{i-1})_{i \in \mathbb{Z}}$ is non-increasing.
\end{proposition}   

\begin{corollary}\label{cor.logconcave}
    If $(a_{i})_{i \in \mathbb{Z}}$ is a log-concave sequence of non-negative real numbers with no internal zeros, then the sequence is trapezoidal and, in particular, it is unimodal.
\end{corollary}

\begin{proof}
   The sequence $(r_i) = (a_i/a_{i-1})$ is non-increasing on the interval where it is well-defined. So for some $-\infty \leq p \leq q \leq + \infty$ one can write: 
   $$
      r_i > 1  \iff i \leq p\,, \qquad \text{ and } \qquad r_i < 1  \iff i \geq q\,.
  $$ 
This precisely means that $(a_{i})$ is trapezoidal and therefore unimodal.
\end{proof}

\begin{remark}
    Corollary~\ref{cor.logconcave} is not true if one does not assume that the sequence does not have internal zeros. Indeed, a sequence can have multiple local maxima and still be log-concave, as long as the ``bumps'' are far apart enough. Consider for example:
    $$
    (a_{i}) = ( \ldots , 0, 1 , 0 , 0 , 1 , 0 , 0 , \ldots )\,.
    $$
\end{remark}

\begin{corollary}\label{generalizedlogconcavity}[Generalized log-concavity inequality]
    Suppose $(a_{i})_{i \in \mathbb{Z}}$ is a log-concave sequence of non-negative real numbers with no internal zeros. For all $i,k \in \mathbb{Z}$, we have:
    $$
    a_{i-k} \, a_{i+k} \leq a_i^2\,.
    $$
\end{corollary}

\begin{proof}
   If any of the three terms is zero, then the inequality obviously holds, because the sequence has no internal zeros. So we can suppose that all three terms are nonzero. In that case, the inequality we wish to prove can be written as follows:
   $$
    \frac{ a_{i+k}}{a_i} \, \leq \frac{a_i}{a_{i-k}}\,.
    $$
But we know that the sequence $(a_i/a_{i-1})$ is non-increasing and thus:
$$
 b_i = \frac{ a_{i}}{a_{i-k}} = \frac{ a_{i}}{a_{i-1}} \times \frac{ a_{i-1}}{a_{i-2}} \times \ldots \times \frac{ a_{i-k+1}}{a_{i-k}}
$$
is the product of $k$ positive non-increasing sequences. So $(b_i)$ is a non-increasing sequence and we have:
\[b_{i+k} \leq b_i \,, \quad \text{ i.e. } \quad \frac{ a_{i+k}}{a_i} \, \leq \frac{a_i}{a_{i-k}}\,.\qedhere\]
\end{proof}

\subsection{Lorentzian polynomials}
We define the \emph{d-th discrete simplex} $\Delta_{m}^{d} \subseteq \mathbb{N}^{m}$ by
$$
\Delta_{m}^{d} = \lbrace{ \alpha \in \mathbb{N}^{m}~\vert~\sum_{i} \alpha_{i} = d \, \rbrace}\,.
$$
For any $1 \leq i \leq m$, we denote by $\epsilon_{i}$ the unit vector in the i-th coordinate.
\par
The notation $w^{\alpha}$ stands for $\prod\limits_{i=1}^{m} w_{i}^{\alpha_{i}}$ and $\partial^{\alpha}=\partial_{1}^{\alpha_{1}}\cdots \partial_{m}^{\alpha_{m}}$. We also introduce $\alpha!=\prod\limits_{i=1}^{m}\alpha_{i}!$.

\begin{definition}\label{defn:Lorentzian}
An $m$-variable homogeneous polynomial $P$ of degree $d$ with nonnegative real coefficients is \emph{strictly Lorentzian} if for any $\alpha \in \Delta_{m}^{d-2}$, $\partial^{\alpha} P$ is a nondegenerate quadratic form of signature $(1,m-1)$.
\par
\end{definition}

\begin{definition}
    A subset $\Sigma$ of $\mathbb{N}^{m}$ is \emph{$\mathrm{M}$-convex} if for any index $i$ and any $\alpha, \beta \in \Sigma$ whose $i$-th coordinates satisfy $\alpha_i > \beta_i$, there exists an index $j$ such that
    $$
    \alpha_j < \beta_j \, , \quad \alpha - \epsilon_i + \epsilon_j \in \Sigma \, , \quad \beta - \epsilon_j + \epsilon_i \in \Sigma \, .
    $$
\end{definition}

The set of degree $d$ homogeneous polynomials in $m$ variables with nonnegative coefficients and an $\mathrm{M}$-convex support is denoted by $\mathrm{M}_m^d$.

\begin{definition}
 A homogeneous real polynomial is \emph{Lorentzian} if it is the limit of strictly Lorentzian polynomials of the same degree and in the same number of variables.   
\end{definition}

Br{\"a}nd{\'e}n and Huh \cite{branden2020lorentzian} prove that in the space of degree $d$ homogeneous polynomials in $m$ variables with real coefficients, the set of Lorentzian polynomials $\mathrm{L}_m^d$ can be described completely: $\mathrm{L}_m^2$ is the set of degree $2$ homogeneous polynomials with nonnegative coefficients and at most one positive eigenvalue. For $d > 2$, we have
$$
\mathrm{L}_m^d = \{ P \in \mathrm{M}_m^d : \partial^{\alpha} P \in \mathrm{L}_m^2 \text{ for any $\alpha \in \Delta_{m}^{d-2}$}  \} \,.
$$

\begin{definition}\label{defn:Normalization}
Given a homogeneous polynomial of degree $d$ in $m$ variables $P=\sum\limits_{\alpha \in \Delta_{m}^{d}} c_{\alpha}w^{\alpha}$, we define its \emph{normalization} $\normal(P)$ by setting $$\normal(P) = \sum\limits_{\alpha \in \Delta_{m}^{d}} \frac{c_{\alpha}w^{\alpha}}{\alpha!}\,.$$
\end{definition}

As shown in Proposition 4.4 of \cite{branden2020lorentzian}, we have the following.

\begin{proposition}\label{prop:LorentzianImpliesUltraLogConcavity}
Given a homogeneous polynomial $P=\sum\limits_{\alpha \in \Delta_{m}^{d}} c_{\alpha}w^{\alpha}$, if $\normal(P)$ is Lorentzian, then the coefficients of $P$ satisfy
$c_{\alpha+\epsilon_{i}-\epsilon_{j}}c_{\alpha-\epsilon_{i}+\epsilon_{j}}\leq c_{\alpha}^{2} $.
\end{proposition}

Let us state a few stability properties satisfied by Lorentzian polynomials.

\begin{proposition}[{\cite[Corollary~3.8]{branden2020lorentzian}}]\label{lorentzproduct}
Given two homogeneous polynomials $P$ and $Q$, if $\normal(P)$ and $\normal(Q)$ are Lorentzian, then $\normal(PQ)$ is also Lorentzian.
\end{proposition}

Note that the converse does not hold. For example, let $f = 6x^3 + 2xy^2+6ay^3$ with $\normal(f)=x^3+xy^2+ay^3$. Then $\normal(f^2)$ is Lorentzian, but for $a\in (-\frac{\sqrt{6}}{18}, 0)$, $\normal(f)$ is not Lorentzian.

\begin{proposition}\label{prop.LaurentShift}
    If $P(w_1,\dots, w_m)$ is a homogeneous polynomial and $w_k^{-1}P(w_1,\dots, w_m)$ is also a polynomial, then $\normal(P)$ is Lorentzian if and only if $\normal(w_k^{-1}P)$ is Lorentzian.
\end{proposition}
\begin{proof}
    The reverse implication follows from Proposition \ref{lorentzproduct}. To prove the forward implication, assume $\normal(P)$ is Lorentzian and write $P=\sum\limits_{\alpha \in \Delta_{m}^{d}} c_{\alpha}w^{\alpha}$. Then 
    \[
    \normal(w_k^{-1}P)=
    \sum\limits_{\alpha \in \Delta_{m}^{d}} \frac{c_{\alpha}w^{\alpha-\e_k}}{(\a-\e_k) !}\,.
    \]
    By assumption $c_\a=0$ whenever $\a_k=0$, thus the above sum is well-defined. Equivalently, we assume that $\a\geq e_k$. Note that $\normal(w_k^{-1}P)$ is of degree $d-1$. Fix $\g\in \Delta_m^{d-3}$. Then 
    \[
    {\partial^\g}
    [\normal(w_k^{-1}P)]
    =
    \sum\limits_{\substack{\alpha \in \Delta_{m}^{d}\\\a\geq\g+\e_k}} \frac{c_{\alpha}w^{\alpha-\e_k-\g}}{(\a-\e_k-\g) !}
    =
    \sum\limits_{\substack{\b \in \Delta_{m}^{2}}} \frac{c_{\g+\e_k+\b}w^{\b}}{\b!}
    \]
    is a degree two polynomial. The matrix entries of the Hessian $\tilde H^\g$ of ${\partial^\g}
    [\normal(P)]$ are $(\tilde H^\g)_{ij}=c_{\g+\e_k+\e_i+\e_j}$.

    A similar computation shows that the Hessian $H^{\g+\e_k}$ of $\partial^{\g+\e_k}[\normal(P)]$ satisfies $H^{\g+\e_k}=\tilde H^\g$. Since $H^{\g+\e_k}$ has at most one positive eigenvalue, so does $\tilde H^\g$ and $\normal(w_k^{-1}P)$ is Lorenztian.
\end{proof}

In addition, the next result has been proved as Lemma~4.8 in \cite{branden2023lower}.

\begin{proposition}
    Set $P(w_1, \ldots,  w_m)$ a homogeneous polynomial and $Q(w_1, \ldots, w_{m-1}) := P(w_1, \ldots, w_{m-1}, w_{m-1})$. If $\normal(P)$ is Lorentzian, then so is $\normal(Q)$.
\end{proposition}


\section{Proof of Propositions~\ref{conj.4} and \ref{algebraic lemma}}\label{sec:Proof}

We recall some ideas developed by Crowell \cite{Cro}. Fix an alternating link $L$ with alternating diagram $D$. Because $D$ is alternating, coloring $S^2$ following the rule shown in Figure~\ref{fig:FIG4} produces a coherent checkerboard coloring of the plane. 

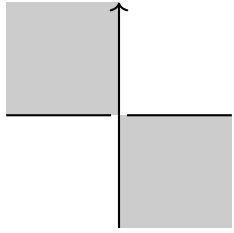
\begin{figure}[h!]
\begin{tikzpicture}[scale=1.5]
    \draw[] (-1,0) to (1,0);
    \draw[->, over] (0,-1) to (0,1);
    \path[fill=black,opacity=0.2] (-1,0) --  (0,0) -- (0,1) -- (-1,1) -- cycle;
    \path[fill=black,opacity=0.2] (1,0) --  (0,0) -- (0,-1) -- (1,-1) -- cycle;
\end{tikzpicture}
\caption{When you travel along the orientation of $D$ and go over a crossing, the north west and south east regions are colored in black.}\label{fig:FIG4}
\end{figure} 
Moreover, if one applies the transformation shown in Figure~\ref{fig:FIG5}, one defines an orientation $o$ on the four-valent graph underlying $D$ which differs from the orientation inherited from $L$. Denote this oriented graph by $P$ and let $w$ denote the map which assigns a weight to each edge of $P$, still following Figure~\ref{fig:FIG5}.

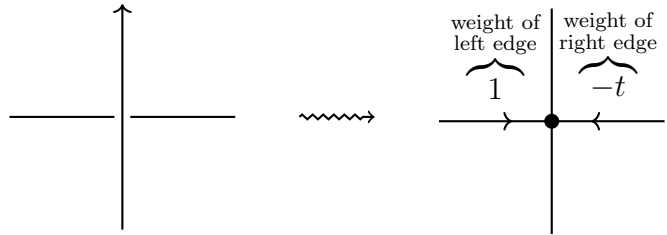
\begin{figure}[h!]
\[
\begin{tikzpicture}[scale=1.5]
    \draw[] (-1,0) to (1,0);
    \draw[->, over] (0,-1) to (0,1);
\end{tikzpicture}
\qquad\begin{tikzpicture}
    \draw[longleadsto,-to] (0,0) to (1,0);
\end{tikzpicture} \qquad
\begin{tikzpicture}[scale=1.5]
    \draw[point=.65] (-1,0) to (0,0);
    \draw[point=.65] (1,0) to (0,0);
    \draw[] (0,-1) to (0,1);
    \node[above] at (-.5,.1) {$\overbrace{1}^{\substack{\text{weight of} \\ \text{left edge}}}$};
    \node[above] at (.5,.1) {$\overbrace{-t}^{\substack{\text{weight of} \\ \text{right edge}}}$};
    \node[circle,fill,inner sep=2pt] at (0,0) {};
\end{tikzpicture}
\]
\caption{Definition of the orientation and edge weights of $P$.}\label{fig:FIG5}
\end{figure} 

For any vertex $i$ of $P$, we define $\mathrm{Tree}(i)$ to be the set of all rooted trees of $P$ with respect to the orientation $o$ and with origin $i$ containing all the vertices of $P$. Note that a \emph{rooted tree} is a subgraph that forms a tree and no two edges share the same terminal vertex.

\begin{theorem}[\cite{Cro}]
    For any $i$ vertex of $P$, the Alexander polynomial of an alternating link $L$ can be expressed as follows:
    $$
    \Delta_L(t) = \sum\limits_{T \in \mathrm{Tree}(i)} \left( \prod\limits_{e \, \mathrm{edge} \, \mathrm{of} \, T} w(e) \right).
    $$
\end{theorem}

For $T$ a rooted tree, one can define $char(T)$ as the number of edges $e$ of $T$ that have weight $w(e) = -t$. In other words, 
$$
char(T) = \# (w^{-1}(-t) \cap T)\,.
$$

\begin{proposition}[\cite{Cro}] \label{maxdeg}
    For $L$ an alternating link, if $
    \Delta_L(t) = \sum\limits_{k = p}^{q} a_k \, t^k
    $ with $a_p , a_q \neq 0$, then $q-p = v-f+1$, where $v$ is the number of vertices and $f$ is the number of Seifert circles of $D$.
\end{proposition}

\begin{proposition}
    Up to an invertible element of $\mathbb{Z}[t^{\pm 1}]$, we have for any alternating link $L$:
    $$
    \Delta_L(t) \, \dot{=} \sum\limits_{k = 0}^{v-f+1} \# \{ T \in \mathrm{Tree}(i) : char(T) = k\} \, (-t)^k .
    $$
\end{proposition}

\begin{proof} The proposition is a consequence of the following computation:
\begin{align*}
  \Delta_L(t)  & = \sum\limits_{T \in \mathrm{Tree}(i)} \left( \prod\limits_{e \, \mathrm{edge} \, \mathrm{of} \, T} w(e) \right) \\
  & =  \sum\limits_{k = p}^{q} \# \{ T \in \mathrm{Tree}(i) : char(T) = k\} \, (-t)^k \\
  & \, \, \dot{=} \sum\limits_{k = 0}^{q-p} \# \{ T \in \mathrm{Tree}(i) : char(T) = k\} \, (-t)^k \\
  & \, \, {=} \sum\limits_{k = 0}^{v-f+1} \# \{ T \in \mathrm{Tree}(i) : char(T) = k\} \, (-t)^k \text{ (see Proposition~\ref{maxdeg})}\,. \qedhere
\end{align*}    
\end{proof}
In addition, Crowell proves in \cite{Cro} that for $0 \leq k \leq v-f+1$, the set $\{ T \in \mathrm{Tree}(i) : char(T) = k\}$ is non empty. Therefore, the latter proves Proposition~\ref{conj.4} modulo Proposition~\ref{algebraic lemma}, and we now wish to prove Proposition~\ref{algebraic lemma}. Before that we recall some definitions.

\begin{definition}\label{def.homog}
    Let $P$ be an $m$-variable polynomial of degree $d$ $$ P(w_1, \ldots, w_{m-1}, w_m) = \sum\limits_{i_1, i_2, \ldots, i_m} c_{i_1, i_2, \ldots, i_m} \prod\limits_{k=1}^{m} w_{i_k}^{\alpha_{i_k}}\,.$$
    The \emph{homogenization of $P$} is the homogeneous polynomial of degree $d$
    $$\mathrm{Homog}(P)(w_1, \ldots, w_{m-1}, w_m, z) = \sum\limits_{i_1, i_2, \ldots, i_m} c_{i_1, i_2, \ldots, i_m} \left( \prod\limits_{k=1}^{m} w_{i_k}^{\alpha_{i_k}} \right) z^{d-\alpha_{i_1} - \ldots - \alpha_{i_m}}.$$
\end{definition}

\begin{definition}
    A homogeneous polynomial $P$ is called \emph{denormalized Lorentzian} if $\normal(P)$ is Lorentzian.
\end{definition}

\begin{proof}[Proof of Proposition~\ref{algebraic lemma}]
We write 
$$
P(t) = \sum\limits_{k = 0}^{d} c_k\, t^k \,, \quad x_k \geq 0 \,, \quad x_d \neq 0 \,.
$$ 
To translate the Lorentzian condition in terms of coefficients of the polynomial, we can now compute the partial derivatives of $$Q(t,z) = \normal\circ\mathrm{Homog}(P) = \sum\limits_{k = 0}^{d} \frac{c_k\, t^k z^{d-k}}{k!(d-k)!}  \, .$$
Set $a,b \in \mathbb{Z}$ such that $0 \leq a \leq d-2$ and $a + b = d-2$. Then
\begin{align*}
 \frac{\partial^a Q}{\partial t^a} (t,z)   
  & =  \sum\limits_{k = a}^{d} \frac{c_k}{(k-a)!(d-k)!} \, t^{k-a} z^{d-k}  
\end{align*}
and
\begin{align*}
 \frac{\partial^{a+b} Q}{\partial z^b \partial t^a} (t,z)   
  & =  \sum\limits_{k = a}^{a+2} \frac{c_k}{(k-a)!(a+2-k)!} \, t^{k-a} z^{a+2-k}
  = \frac{c_a}{2} z^2 + c_{a+1} t z + \frac{c_{a+2}}{2} t^2 \,.
\end{align*}
The Hessian matrix for $\frac{\partial^{a+b} Q}{\partial z^b \partial t^a} (t,z)$ relative to the derivatives $\partial z, \partial t$ is:
\[
\begin{pmatrix}
    c_a                          & c_{a+1} \\
 c_{a+1}                      & c_{a+2}
\end{pmatrix}
\]
with characteristic polynomial and discriminant
\begin{align*}
 \chi_a(\lambda)   
   =  \lambda^2 - (c_a + c_{a+2}) \lambda + c_a c_{a+2} - c_{a+1}^2 
   \qquad 
   \mbox{and}
   \qquad 
   \Delta_a = (c_a - c_{a+2})^2 + 4 c_{a+1}^2 \geq 0 \,.
\end{align*}
The Lorentzian property for $Q(t,z)$ means that the smaller root of $\chi_a$ is nonpositive. In other words, $c_a + c_{a+2} - \sqrt{\Delta_a} \leq 0$ for any $0 \leq a \leq d-2$. Since all the $c_i$ are nonnegative, this last inequality is equivalent to $c_a c_{a+2} \leq c_{a+1}^2$, which exactly means that the sequence $(c_i)$ of coefficients of $P$ is log-concave. This proves Proposition~\ref{algebraic lemma}.
\end{proof}

Proposition~\ref{conj.4} shines a new light on the Fox-Stoimenow conjecture. Indeed, it suggests that a reasonable strategy to attempt to prove this conjecture in full generality is to further pursue the Hafner--M\'esz\'aros--Vidinas approach \cite{Hafner}. From Proposition \ref{prop.LaurentShift}, the Lorentzian property of the denormalized Alexander polynomial is independent of its presentation as an honest polynomial.

Moreover, assuming that the Fox--Stoimenow conjecture is true, Proposition~\ref{lorentzproduct} implies that the homogenization of the square of the Alexander polynomial $\mathrm{Homog}(\Delta_L(-t)^2)$ for an alternating link $L$ is also denormalized Lorentzian, and in turn its sequence of coefficients is log-concave, see Proposition~\ref{prop:LorentzianImpliesUltraLogConcavity}. Recall that the square of the Alexander polynomial is a specialization of the Links--Gould invariant of links, see Equation~\eqref{LGspec}. This observation suggests that some form of the Fox--Stoimenow conjecture lifts to $\LG$, which we formulate in Conjectures~\ref{conj.1}, \ref{conj.2}, and \ref{conj.3}. 

Regarding the Lorentzian property for $\LG$, note that it is sensitive to the choice of variables used to write the polynomial. Although the tests we carried out have failed for the obvious choices of variables, see Subsection \ref{subs.noLG}, it is our prediction that $\LG$ is Lorentzian for some choice of variables yet to be determined. Indeed, the 2d log-concavity property that was tested successfully shows that $\LG$ displays a very structured behavior on alternating links.

\section{Evidence supporting Conjectures~\ref{conj.1}, \ref{conj.2} and \ref{conj.3}}\label{sec:Evidence}

\subsection{The conjectures hold for all knots from the LG-explorer} We wrote a \texttt{Python} program that analyzes the data from De Wit's $\LG$-explorer \cite{LGexplorer} to show that conjectures~\ref{conj.1}, \ref{conj.2} and \ref{conj.3} are true for all prime knots with up to 12 crossings and a selection of prime knots with 13 to 16 crossings and a small braid index, for which the value of the Links-Gould polynomial is available in the $\LG$-explorer database. The results of the computations and verifications as well as the computer program are available online \cite{data}. See Figure~\ref{fig:FIG6} for an example of a graphical representation of the data.

\begin{figure}[hb!]
\includegraphics[scale=.3]{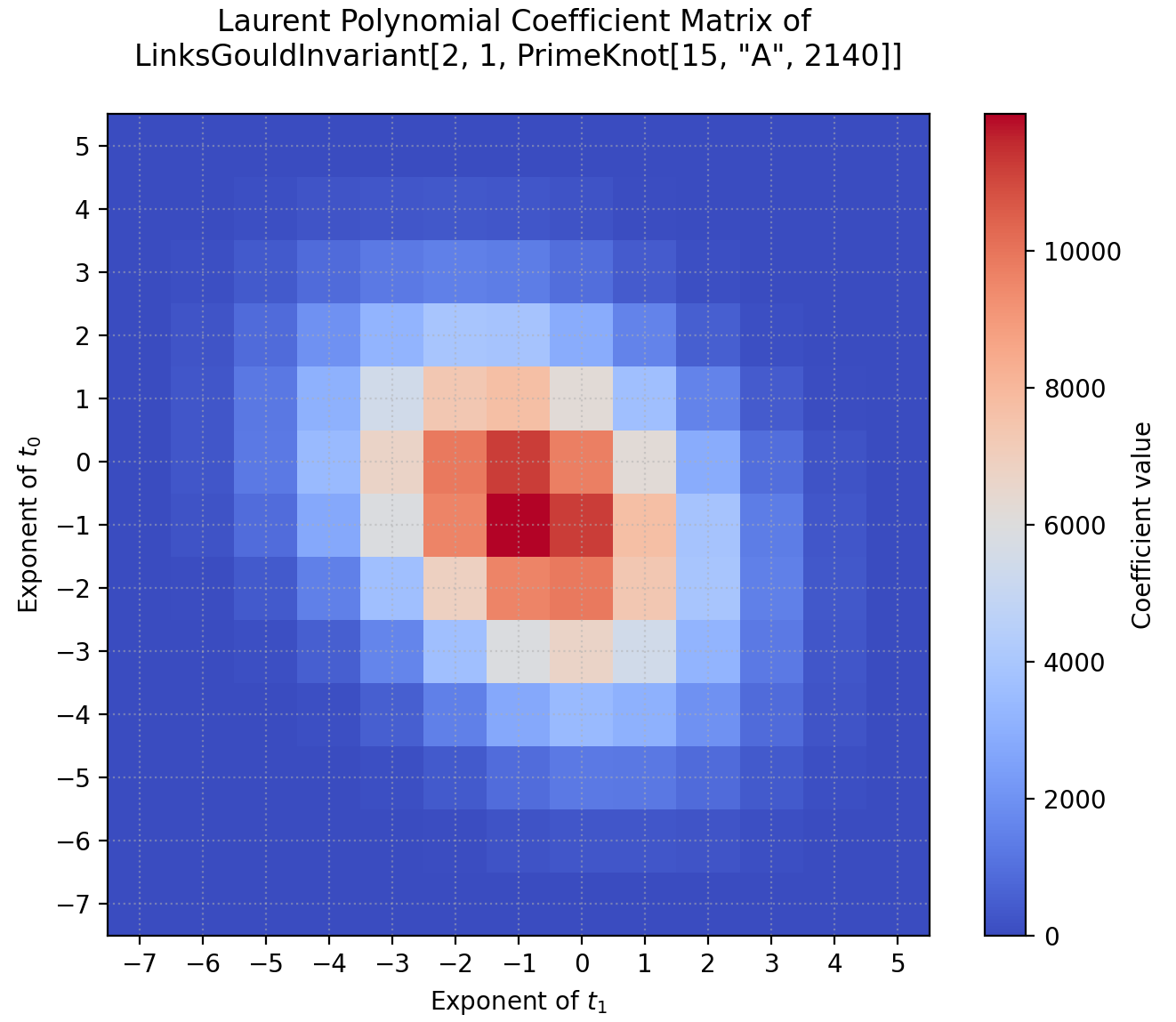}

\caption{A graphical representation of the absolute values of the coefficients of the $\LG$ polynomial of the alternating knot $15_{2140}^a$. The ``mountain'' formed by the coefficients appears clearly.}\label{fig:FIG6}
\end{figure} 

Note that there is an inconsistency in the $\LG$-explorer database between knots with at most 12 crossings and those with more than 12 crossings. Indeed, once you make the usual substitution 
$$
(t_0, t_1)= (q^{-2\a}, q^{2\a + 2})
$$
the polynomials from the dataset are $\LG(K; t_{0}, t_{1})$ when $K$ has at least 13 crossings, but for knots with 12 crossings or less the knot polynomials are $\LG(K; t_{0}^{-1}, t_{1}^{-1})$.

\subsection{The conjectures hold for all \texorpdfstring{$(2,n)$}{(2,n)}-torus links} Setting $n \in \mathbb{Z}_{n\geq 1}$, the \texorpdfstring{$(2,n)$}{(2,n)}-torus link $T(2,n)$ is the closure of the braid $b_n = \sigma_1^n \in B_2$. The \texorpdfstring{$(2,3)$}{(2,3)}-torus link $T(2,3)$ is represented in Figure~\ref{fig:FIG7}. When $n$ is odd, $T(2,n)$ is a knot; when $n$ is even, $T(2,n)$ is a link with two components. In both cases, $T(2,n)$ is alternating. 

\begin{figure}[h!]
\begin{tikzpicture}
    \foreach \x in {0,1,2}
    {
    \draw[->] (1,\x) to[out=90,in=-90] (0,\x+1);
    \draw[over,->] (0,\x) to[out=90,in=-90] (1,\x+1);
    }
    \draw (1,3) arc(180:0:.5);
    \draw (0,3) arc(0:180:.5);
    \draw (0,0) arc(0:-180:.5);
    \draw (1,0) arc(-180:0:.5);
    \draw (-1,0) to (-1,3);
    \draw (2,0) to (2,3);
\end{tikzpicture}
\caption{Torus knot $T(2,3)$ is alternating.}\label{fig:FIG7}
\end{figure}

\begin{lemma}\label{valuetorus}
    For $n \in \mathbb{Z}_{n\geq 1}$:
    $$
    \LG(T(2,n)) = \red{(-1)^{n-1}} + \sum\limits_{k = 2}^{n} \left( \green{t_1^{n-k} t_0 \sum\limits_{l = 0}^{k-2} t_0^l (-1)^{k-2-l}} - \blue{t_1^{n-k+1} \sum\limits_{m = 0}^{k-1} t_0^m (-1)^{k-1-m}} \right)
    $$
\end{lemma}

\begin{proof}
    We use the first skein relation from \cite[Cor.3.2]{Ishii2} to simplify $\sigma_1^n$ for $n\geq 2$:
\begin{align*}
    &\LG\left(~
\begin{tikzpicture}[scale=.75]
    \foreach \x in {0,1.5}
    {
    \draw[->] (1,\x) to[out=90,in=-90] (0,\x+1);
    \draw[over,->] (0,\x) to[out=90,in=-90] (1,\x+1);
    }
    \node at (.5,1.5) {$\vdots$};
    \draw [very thick, decorate,decoration={brace,mirror,amplitude=10pt},xshift=1em] (1,0) -- (1,2.5) node [black,midway,xshift=4em] {$n$ half twists};
\end{tikzpicture}
\right)
\\&=
\left(
\frac{(-1)^n}{(t_0+1)(t_1+1)}
+
\frac{t_0^n}{(t_0+1)(t_0-t_1)}
+
\frac{t_1^n}{(t_1+1)(t_1-t_0)}
\right)
\LG\left(
\begin{tikzpicture}[scale=.75]
    \draw[->] (1,0) to[out=90,in=-90] (0,1);
    \draw[over,->] (0,0) to[out=90,in=-90] (1,1);
    \draw[->] (1,1) to[out=90,in=-90] (0,2);
    \draw[over,->] (0,1) to[out=90,in=-90] (1,2);
\end{tikzpicture}
\hspace{-1em}
\right)
\\
&-
\left(
\frac{(-1)^n(t_0+t_1)}{(t_0+1)(t_1+1)}
+
\frac{t_0^n(t_1-1)}{(t_0+1)(t_0-t_1)}
+
\frac{t_1^n(t_0-1)}{(t_1+1)(t_1-t_0)}
\right)
\LG\left(
\begin{tikzpicture}[scale=.75]
\begin{scope}
    \draw[->] (1,0) to[out=90,in=-90] (0,1);
    \draw[over,->] (0,0) to[out=90,in=-90] (1,1);
\end{scope}
\end{tikzpicture}
\hspace{-1em}
\right)
\\
&+
\left(
\frac{(-1)^nt_0t_1}{(t_0+1)(t_1+1)}
-
\frac{t_0^nt_1}{(t_0+1)(t_0-t_1)}
-
\frac{t_0t_1^n}{(t_1+1)(t_1-t_0)}
\right)
\LG\left(~
\begin{tikzpicture}[scale=.75]
    \draw[->] (0,0) to (0,1);
    \draw[->] (.75,0) to (.75,1);
\end{tikzpicture}
~
\right)\,.
\end{align*}
Applying the partial trace to these tangles to evaluate $\LG$ gives 
\begin{align*}
    LG(T(2,n))=&\left(
\frac{(-1)^n}{(t_0+1)(t_1+1)}
+
\frac{t_0^n}{(t_0+1)(t_0-t_1)}
+
\frac{t_1^n}{(t_1+1)(t_1-t_0)}
\right)\cdot (t_0+t_1 -t_0t_1-1)
\\&-
\left(
\frac{(-1)^n(t_0+t_1)}{(t_0+1)(t_1+1)}
+
\frac{t_0^n(t_1-1)}{(t_0+1)(t_0-t_1)}
+
\frac{t_1^n(t_0-1)}{(t_1+1)(t_1-t_0)}
\right)
\end{align*}
where $(t_0+t_1 -t_0t_1-1)$ is the value on the Hopf link and the trace of the identity resolution vanishes. Then we reduce the expression in order to clear the denominators. We rearrange the terms and with some patience we get the expected formula.
\end{proof}

\begin{proposition}
Conjectures~\ref{conj.1}, \ref{conj.2}, and \ref{conj.3} hold for all \texorpdfstring{$(2,n)$}{(2,n)}-torus links.
\end{proposition}

\begin{proof}
    From the expression for $\LG(T(2,n))$ given in Lemma~\ref{valuetorus}, it follows at once that Conjecture \ref{conj.1} is satisfied. Expanding on that, we draw the graph representing the absolute value of the coefficient in front of the monomial $t_0^i t_1^j$ as a function of $(i,j)$, see Figure~\ref{fig:FIG8}. The graph is obtained by adding the different ``layers'' of coefficients in the formula from Lemma~\ref{valuetorus}: the \green{green} layer, the \blue{blue} layer, and the isolated \red{red} coefficient. Figure~\ref{fig:FIG8} shows clearly that Conjecture \ref{conj.3} is true.

    To prove Conjecture \ref{conj.2}, we want to show that the mountain of absolute values $|a_{i,j}|$ of coefficients of $\LG(T(2,n))$ is 2d log-concave. The only possible values for $|a_{i,j}|$ are in the set $\{0,1,2\}$. If $a,b,c \in \{0,1,2\}$ with $c \neq 0$, which we can suppose since $(a_{i,j})$ has no interior zeros, are three coefficients of $\LG(T(2,n))$, let us see that $ab \leq c^2$. \\
$\bullet$ \ovalbox{$(a,b,c) = (0,0,0)$}\,:  $0 \times 0 \leq 0^2 = 0$;\\
$\bullet$ \ovalbox{$(a,b,c) = (0,0,1)$}\,:   $0 \times 0 \leq 1^2 = 1$;\\
$\bullet$ \ovalbox{$(a,b,c) = (0,0,2)$}\,:   $0 \times 0 \leq 2^2 = 4$;\\
In fact, if $a$ \textit{or} $b$ is zero, the computation is similar to one of the previous cases. So the remaining cases to study are when $a \neq 0$ \textit{and} $b \neq 0$. \\
$\bullet$ \ovalbox{$(a,b) = (1,1)$}\,:   Then $c = 1$ or $c = 2$ because of Figure~\ref{fig:FIG8}, and the inequality holds;\\
$\bullet$ \ovalbox{$(a,b) = (1,2)$ or $(2,1)$}\,: then $c = 2$ like in the previous case, and the inequality is \mbox{satisfied;}\\
$\bullet$ \ovalbox{$(a,b) = (2,2)$}\,:  once again $c = 2$ because of Figure~\ref{fig:FIG8}. So the inequality is also verified in this last case. \\

This concludes the proof of Conjecture \ref{conj.2} for \texorpdfstring{$(2,n)$}{(2,n)}-torus links.
\end{proof}

\subsection{The conjectures hold for all twist knots}

A twist knot is a Whitehead double of the unknot. We will denote by $K_n$ the twist knot shown in Figure~\ref{fig:FIG9} when $2n - 1$ is positive. If $2n - 1$ is negative, there are $1 - 2n$ crossings of the opposite sort.

\begin{figure}[h!]
\begin{tikzpicture}  
    \draw (0,0) to[out=225,in=-90] (-1,1);
    \draw[point=.5] (1,0) to[out=-45,in=-90] (2,1);
    \draw (-1,1) to[out=90, in=180] (.5,2.5);
    \draw[over,point=.5] (0,2.5) arc(180:0:.5);
    \draw (1,2.5) to[out=90, in=0] (.5,3);
    \draw[over] (1,0) to[out=135,in=-90] (0,1);
    \draw[over] (0,0) to[out=45,in=-90] (1,1);
    \draw[over] (1,1.5) to[out=90,in=-90] (0,2.5);
    \draw[over] (0,1.5) to[out=90,in=-90] (1,2.5);
    \node at (.5,1.25) {$\vdots$};
    \draw[over] (2,1) to[out=90, in=0] (.5,2.5);
    \draw [very thick, decorate,decoration={brace,mirror,amplitude=10pt},xshift=1em] (1.75,.25) -- (1.75,2.25) node [black,midway,xshift=4em] {\begin{tabular}{c} $2n-1$\\ half twists \end{tabular}};
\end{tikzpicture}
\caption{Twist knot $K_n$.}\label{fig:FIG9}
\end{figure}
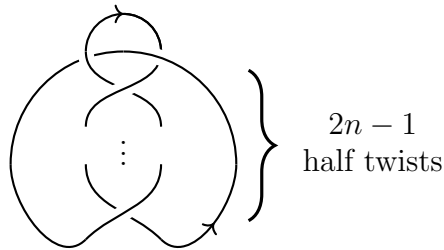

\begin{lemma}\label{valuewist}
    For $n \in \mathbb{Z}$, the value of the Links--Gould polynomial on the twist knot $K_n$ is as follows:\\
    $\bullet$ for $n = 0$:  $\LG(K_0) = \LG(\bigcirc) = 1 ;$\\
    $\bullet$ for $n \leq -1$, setting $l = -n$ :
    \begin{align*}
\LG(K_n)
& = \red{1} + \green{(2-t_0-t_1-t_0^{-1}-t_1^{-1} + t_0 t_1 + t_0^{-1} t_1^{-1}) \sum\limits_{k = 0}^{l-1} (t_0 t_1)^k} \\
& + \blue{(4 - 2 t_0 - 2 t_1 - 2 t_0^{-1} - 2 t_1^{-1} + t_0 t_1 + t_0^{-1} t_1^{-1} + t_0 t_1^{-1} + t_0^{-1} t_1) \sum\limits_{k = 0}^{l-1}(2 l - 2 k -1) (t_0 t_1)^k} \,;
\end{align*}
    $\bullet$ for $n \geq 1$, writing the formula in terms of the variables $(t_0^{-1}, t_1^{-1})$ for clarity:

\begin{samepage}    
    \begin{align*}
\LG(K_n)(t_0^{-1}, t_1^{-1})
& = \red{(t_0 t_1)^{n-1}(-t_0^2 t_1 - t_0 t_1^2 + t_0^2 + t_1^2 + 2 t_0 t_1 - t_0 - t_1 +1)} \\
& + \green{(t_0^2t_1^2 - 3 t_0^2 t_1 - 3 t_0 t_1^2 + 2 t_0^2 + 2 t_1^2 + 5 t_0 t_1 - 3 t_0 - 3 t_1 +2) \sum\limits_{k = 0}^{n-2} (t_0 t_1)^k} \\
& + \blue{(t_0^2 t_1^2 - 2 t_0^2 t_1 - 2 t_0 t_1^2 + t_0^2 + t_1^2 + 4 t_0 t_1 - 2 t_0 - 2 t_1 +1) \sum\limits_{k = 0}^{n-2}(2 n - 2 k -3) (t_0 t_1)^k}.
\end{align*}   
\end{samepage}
\end{lemma}

\begin{proof}
    We apply \cite[Prop. 2.18]{Kohli2} to expand and simplify the given expressions.
\end{proof}

By adding the three layers of each formula, shown in different colors, we obtain the graphs in Figures~\ref{fig:FIG10} and \ref{fig:FIG11}.

\begin{proposition}
Conjectures~\ref{conj.1}, \ref{conj.2}, and \ref{conj.3} hold for all twist knots.
\end{proposition}

\begin{proof}
    Conjectures \ref{conj.1} and \ref{conj.3} are straightforward consequences of the formulas from Lemma~\ref{valuewist} and carefully observing Figures~\ref{fig:FIG10} and \ref{fig:FIG11}. The proof of Conjecture \ref{conj.2} is far more tedious. Let us prove that $\LG(K_n)$ satisfies the conjecture for $n \leq -1$ by induction on $l = -n$. The other case works the same way. 

    The knot $K_n = K_{-l}$ has $2 + (1+2l) = 2l + 3$ crossings. So for $l \leq 4$, we know that \ref{conj.2} is verified for $K_{-l}$ because of prior verifications for knots with up to $12$ crossings. Suppose that for some $l \geq 1$, $K_{-l}$ satisfies Conjecture \ref{conj.2}. We show that $K_{-l-1}$ also satisfies \ref{conj.2}. Indeed, using Figure~\ref{fig:FIG10}, the lower-left corner of the table of coefficients for $\LG(K_{-l-1})$ is given in Figure \ref{fig.twistcont}.

\begin{figure}[h!]
    \begin{tikzpicture}[scale=0.8]

\fill[blue!20] (0,0) rectangle (3,3);

\draw[step=1cm,gray,thin] (0,0) grid (7,7);

\draw[->,thick] (0,0) -- (7.5,0) node[right] {$t_0$};

\draw[->,thick] (0,0) -- (0,7.5) node[above] {$t_1$};

\node[font=\tiny,left] at (0,0.5) {$-1$};
\node[font=\tiny,left] at (0,1.5) {$0$};
\node[font=\tiny,left] at (0,2.5) {$1$};
\node[font=\tiny,left] at (0,3.5) {$2$};
\node[font=\tiny,left] at (0,4.5) {$3$};
\node[font=\tiny,left] at (0,5.5) {$4$};

\node[font=\tiny,below] at (0.5,0) {$-1$};
\node[font=\tiny,below] at (1.5,0) {$0$};
\node[font=\tiny,below] at (2.5,0) {$1$};
\node[font=\tiny,below] at (3.5,0) {$2$};
\node[font=\tiny,below] at (4.5,0) {$3$};
\node[font=\tiny,below] at (5.5,0) {$4$};

\node[font=\tiny,align=center] at (0.5,0.5) {$2l$ \\ $+ 2$};
\node[font=\tiny, align=center] at (1.5,0.5) {$4l$ \\ $+ 3$};
\node[font=\tiny, align=center] at (0.5,1.5) {$4l$ \\ $+ 3$};
\node[font=\tiny, align=center] at (0.5,2.5) {$2l$ \\ $+ 1$};
\node[font=\tiny, align=center] at (2.5,0.5) {$2l$ \\ $+ 1$};
\node[font=\tiny, align=center] at (1.5,1.5) {$10l$ \\ $+ 7$};
\node[font=\tiny, align=center] at (2.5,1.5) {$8l$ \\ $+ 2$};
\node[font=\tiny, align=center] at (1.5,2.5) {$8l$ \\ $+ 2$};
\node[font=\tiny, align=center] at (3.5,1.5) {$2l$ \\ $- 1$};
\node[font=\tiny, align=center] at (1.5,3.5) {$2l$ \\ $- 1$};
\node[font=\tiny, align=center] at (2.5,2.5) {$12l$ \\ $- 3$};
\node[font=\tiny, align=center] at (3.5,2.5) {$8l$ \\ $- 6$};
\node[font=\tiny, align=center] at (2.5,3.5) {$8l$ \\ $- 6$};
\node[font=\tiny, align=center] at (3.5,3.5) {$12l$ \\ $- 15$};
\node[font=\tiny, align=center] at (4.5,2.5) {$2l$ \\ $- 3$};
\node[font=\tiny, align=center] at (2.5,4.5) {$2l$ \\ $- 3$};
\node[font=\tiny, align=center] at (3.5,4.5) {$8l$ \\ $- 14$};
\node[font=\tiny, align=center] at (4.5,3.5) {$8l$ \\ $- 14$};
\node[font=\tiny, align=center] at (4.5,4.5) {$12l$ \\ $- 27$};

\node at (3.5,5.6) {$\iddots$};
\node at (4.5,5.6) {$\iddots$};
\node at (5.5,5.6) {$\iddots$};
\node at (5.5,4.6) {$\iddots$};
\node at (5.5,3.6) {$\iddots$};

\end{tikzpicture}
\caption{Coefficients of $\LG(K_{-l-1})$ near the constant term.}
\label{fig.twistcont}
\end{figure}

Outside of the $3 \times 3$ blue square, the rest of the table reproduces part of the table of coefficients for $\LG(K_{-l})$. So any relation needed to satisfy Conjecture \ref{conj.2} that does not involve any of the nine coefficients in the blue square is true by the induction hypothesis. To conclude, we need to show that \emph{all} relations involving at least one of these nine coefficients are true. To do this the only way is to consider all possible cases. A useful tool here is Corollary~\ref{generalizedlogconcavity} that reduces the study of inequalities along any line to proving the inequalities involving three consecutive coefficients, most of which are already hold by induction. We omit the remaining details of verifying the inequalities for brevity. 
    
\end{proof}

\subsection{\texorpdfstring{$\LG$}{LG} is not (denormalized) Lorentzian on alternating links for any obvious choice of variables}
\label{subs.noLG}

Here we consider various choices of variables in the presentation for the Links--Gould polynomial of the trefoil. For each choice, we exhibit a choice of derivative and corresponding quadratic form whose Hessian has more than one positive eigenvalue.

\begin{align*}
    \LG(\mathsf{3_1})(t_0, t_1) &= -t_0^2t_1 + t_0^2 - t_0t_1^2 + 2t_0t_1 - t_0 + t_1^2 - t_1 + 1
\\
\normal\circ\homog(\LG(\mathsf{3_1})(-t_0, -t_1))
&=
\left\lbrace 
\begin{matrix}
    \dfrac{1}{2}t_0^2t_1 + \dfrac{1}{2}t_0^2z + \dfrac{1}{2}t_0t_1^2 + 2t_0t_1z \\[2ex]+ \dfrac{1}{2}t_0z^2 + \dfrac{1}{2}t_1^2z + \dfrac{1}{2}t_1z^2 + \dfrac{1}{3!}z^3
\end{matrix}
\right.
\\
\frac{\partial}{\partial z}[\normal\circ\homog(\LG(\mathsf{3_1})(-t_0, -t_1))]
&=
\frac{1}{2}t_0^2 + 2t_0t_1 + t_0z + \frac{1}{2}t_1^2 + t_1z + \frac{1}{2}z^2
\\
H &= \begin{pmatrix} 1&2&1\\2&1&1\\1&1&1
\end{pmatrix}
\end{align*}
\[\lambda_1= -1,\quad \lambda_2= 2-\sqrt{3},\quad \lambda_3= 2+\sqrt{3}\,.\]
Since this derivative has a Hessian with two positive eigenvalues, the Links--Gould polynomial is not denormalized Lorentzian.

We also consider replacing a variable with a power of itself together with a variable shift such that we obtain a polynomial. 

\begin{align*}
LG(\mathsf{3_1})(t_0,t_1^{-1})  \cdot  t_1^2 &= t_0^2 t_1^2 - t_0^2 t_1 - t_0 t_1^2 + 2 t_0 t_1 - t_0 + t_1^2 - t_1 + 1\\
\normal\circ\homog(\LG(\mathsf{3_1})(t_0,t_1^{-1})) &= \left\lbrace
\begin{matrix}
    \dfrac{1}{4} t_0^2 t_1^2 + \dfrac{1}{2} t_0^2 t_1 z + \dfrac{1}{2} t_0 t_1^2 z + t_0 t_1 z^2
    \\[2ex]
    + \dfrac{1}{6} t_0 z^3 + \dfrac{1}{4} t_1^2 z^2 + \dfrac{1}{6} t_1 z^3 + \dfrac{1}{24} z^4
\end{matrix}\right. \\
\frac{\partial^2}{\partial t_0 \partial z}[\normal\circ\homog(\LG(\mathsf{3_1})(t_0,t_1^{-1}))]&= t_0 t_1 + \frac{1}{2} t_1^2 + 2 t_1 z + \frac{1}{2} z^2
\\
H &= \begin{pmatrix} 0&1&0\\1&1&2\\0&2&1
\end{pmatrix}
\end{align*}
\[\lambda_1 = -1.39138238, \quad \lambda_2= 0.22713444,\quad  
\lambda_3 = 3.16424794\,.\]

Another choice of variables for the Links-Gould polynomial comes from \cite{GK:multi}. We write this polynomial $V_1(p,q)$. Then 
\begin{align*}
    V_{1,\mathsf{3_1}}(p,q)&=p^2q^2 - pq^2 - pq + 2q + 1 - qp^{-1} - p^{-1} + p^{-2}
    \\
    \normal\circ\homog(p^2\cdot V_{1,\mathsf{3_1}}(-p,q))&=\left\lbrace
    \begin{matrix}
        \dfrac{1}{2\cdot 4!}p^4q^2 + \dfrac{1}{2\cdot 3!}p^3q^2z + \dfrac{1}{2\cdot 3!}p^3qz^2+ \dfrac{1}{3!}p^2qz^3
        \\[2ex]
         + \dfrac{1}{2\cdot 4!}p^2z^4 + \dfrac{1}{4!}pqz^4 + \dfrac{1}{5!}pz^5 + \dfrac{1}{6!}z^6
    \end{matrix}  
    \right.
    \\
    \frac{\partial^4}{\partial p^3 \partial z}
    [\normal\circ\homog(p^2\cdot V_{1,\mathsf{3_1}}(-p,q))]
    &=
    pq + qz + \frac{1}{2}z^2
    \\
    H&=\begin{pmatrix} 0&1&0\\1&0&1\\0&1&1
\end{pmatrix}
\end{align*}
\[\lambda_1=-1.24698, 
    \quad 
    \lambda_2=0.445042,
    \quad
    \lambda_3=1.80194\,.\]

There are also Conway polynomial like changes of variables for $V_1$ which desymmetrize $V_1$. We denote the change of variables in $u = p + p^{-1}q^{-1}$ and $q$ by $\tilde V_1(u,q)$. We have
\begin{align*}
    \tilde V_{1,\mathsf{3_1}}(u,q) &= q^2u^2 + (2q^3-q^2+q)u + (q^4-q^3+q^2-q+1)
    \\
    \normal\circ\homog (\tilde V_{1,\mathsf{3_1}}(-u,-q)) &= 
    \left\lbrace
    \begin{matrix}
        \dfrac{1}{2\cdot 2} u^2q^2 + 
    \dfrac{2}{3!}uq^3+\dfrac{1}{2}uq^2z+\dfrac{1}{2}uqz^2 +\dfrac{1}{4!}q^4
    \\[2ex]
    +\dfrac{1}{3!}q^3z+\dfrac{1}{2\cdot 2}q^2z^2+\dfrac{1}{3!}qz^3+\dfrac{1}{4!}z^4
    \end{matrix}
    \right.
    \\
    \frac{\partial^2}{\partial u \partial q }[\normal\circ\homog (\tilde V_{1,\mathsf{3_1}}(-u,-q))]
    &=uq+q^2+qz+\frac{1}{2}z^2
    \\
    H&=\begin{pmatrix} 0&1&0\\1&2&1\\0&1&1
\end{pmatrix}
\end{align*}
\[\lambda_1=-0.532089,\quad
\lambda_2= 0.652704,\quad
\lambda_3= 2.87939\,.\]
We also set $\hat V_1(v,q)$ to be the polynomial in the variables $v = p + p^{-1}q^{-1} - q - q^{-1}$ and $q$. Then
\begin{align*}
    \hat V_{1,\mathsf{3_1}}(v,q) &= v^2q^2 - (q^2+q)v
    \\
    \normal\circ\homog (\hat V_{1,\mathsf{3_1}}(-v,q)) &= \frac{1}{2\cdot 2}v^2q^2  +\frac{1}{2}vq^2z+\frac{1}{2}vqz^2
    \\
    \frac{\partial^2}{\partial v \partial q}[\normal\circ\homog (\hat V_{1,\mathsf{3_1}}(-v,q))]&=vq + qz + \frac{1}{2}z^2
\end{align*}
\[H=\begin{pmatrix} 0&1&0\\1&0&1\\0&1&1
\end{pmatrix}\]
\[\lambda_1=-1.24698, 
    \quad 
    \lambda_2=0.445042,
    \quad
    \lambda_3=1.80194\,.\]
Which is the same Hessian as $\LG(\mathsf{3_1})(t_0,t_1)$.

\section{Values of Links--Gould on some classes of links}\label{sec.values}
We present the coefficients of the Links--Gould polynomial on $(2,n)$-torus links, and on twist knots in the variables $(t_0,t_1)$ and $(t_0^{-1},t_1^{-1})$.

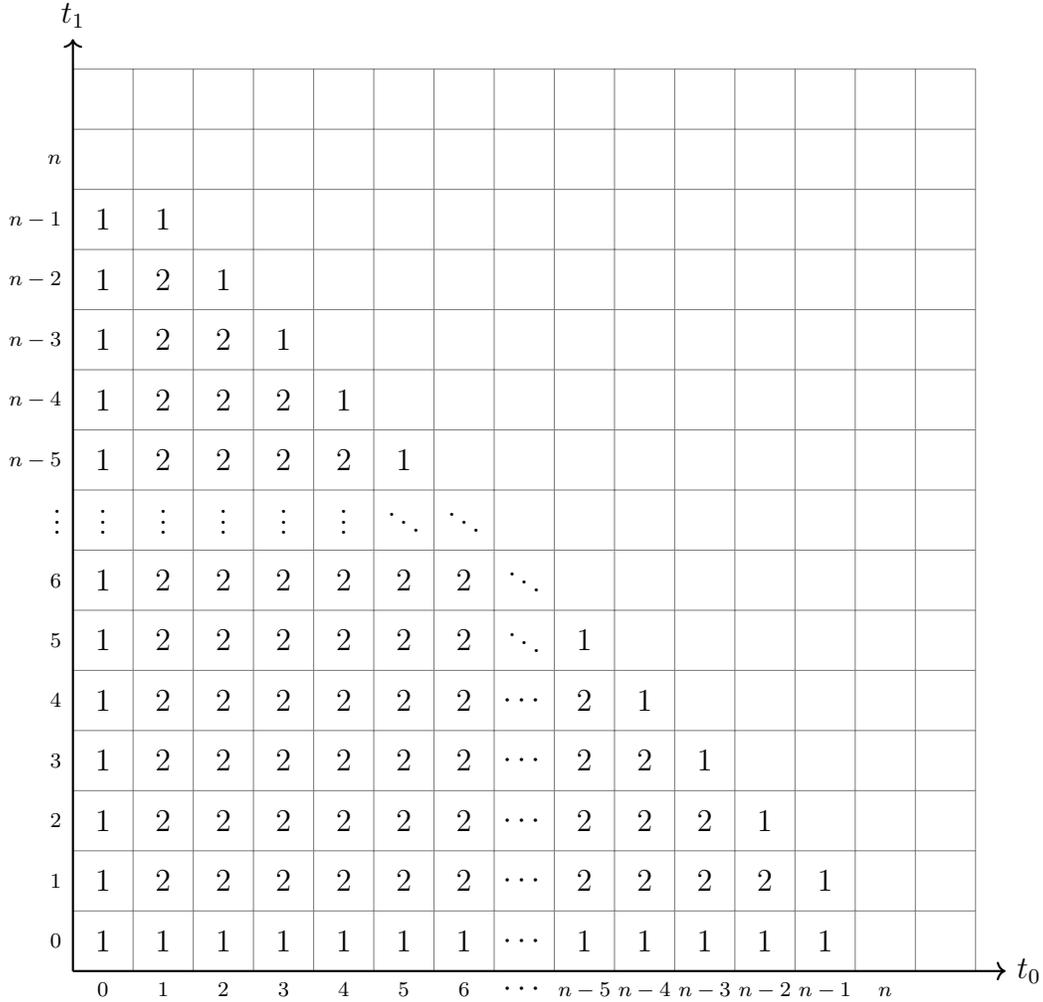
\begin{figure}[h!]
\begin{tikzpicture}[scale=0.8]

\draw[step=1cm,gray,thin] (0,0) grid (15,15);

\draw[->,thick] (0,0) -- (15.5,0) node[right] {$t_0$};

\draw[->,thick] (0,0) -- (0,15.5) node[above] {$t_1$};

\node[font=\tiny,left] at (0,0.5) {$0$};
\node[font=\tiny,left] at (0,1.5) {$1$};
\node[font=\tiny,left] at (0,2.5) {$2$};
\node[font=\tiny,left] at (0,3.5) {$3$};
\node[font=\tiny,left] at (0,4.5) {$4$};
\node[font=\tiny,left] at (0,5.5) {$5$};
\node[font=\tiny,left] at (0,6.5) {$6$};
\node[left] at (0,7.6) {$\vdots$};
\node[font=\tiny,left] at (0,8.5) {$n-5$};
\node[font=\tiny,left] at (0,9.5) {$n-4$};
\node[font=\tiny,left] at (0,10.5) {$n-3$};
\node[font=\tiny,left] at (0,11.5) {$n-2$};
\node[font=\tiny,left] at (0,12.5) {$n-1$};
\node[font=\tiny,left] at (0,13.5) {$n$};

\node[font=\tiny,below] at (0.5,0) {$0$};
\node[font=\tiny,below] at (1.5,0) {$1$};
\node[font=\tiny,below] at (2.5,0) {$2$};
\node[font=\tiny,below] at (3.5,0) {$3$};
\node[font=\tiny,below] at (4.5,0) {$4$};
\node[font=\tiny,below] at (5.5,0) {$5$};
\node[font=\tiny,below] at (6.5,0) {$6$};
\node[below] at (7.5,0) {$\cdots$};
\node[font=\tiny,below] at (8.5,0) {$n-5$};
\node[font=\tiny,below] at (9.5,0) {$n-4$};
\node[font=\tiny,below] at (10.5,0) {$n-3$};
\node[font=\tiny,below] at (11.5,0) {$n-2$};
\node[font=\tiny,below] at (12.5,0) {$n-1$};
\node[font=\tiny,below,align=center] at (13.5,0) {$n$\vphantom{1}};

\node at (0.5,0.5) {1};
\node at (0.5,1.5) {1};
\node at (0.5,2.5) {1};
\node at (0.5,3.5) {1};
\node at (0.5,4.5) {1};
\node at (0.5,5.5) {1};
\node at (0.5,6.5) {1};

\node at (0.5,8.5) {1};
\node at (0.5,9.5) {1};
\node at (0.5,10.5) {1};
\node at (0.5,11.5) {1};
\node at (0.5,12.5) {1};
\node at (1.5,12.5) {1};
\node at (2.5,11.5) {1};
\node at (3.5,10.5) {1};
\node at (4.5,9.5) {1};
\node at (5.5,8.5) {1};

\node at (8.5,5.5) {1};
\node at (9.5,4.5) {1};
\node at (10.5,3.5) {1};
\node at (11.5,2.5) {1};
\node at (12.5,1.5) {1};
\node at (12.5,0.5) {1};
\node at (11.5,0.5) {1};
\node at (10.5,0.5) {1};
\node at (9.5,0.5) {1};
\node at (8.5,0.5) {1};

\node at (6.5,0.5) {1};
\node at (5.5,0.5) {1};
\node at (4.5,0.5) {1};
\node at (3.5,0.5) {1};
\node at (2.5,0.5) {1};
\node at (1.5,0.5) {1};

\node at (1.5,1.5) {2};
\node at (2.5,1.5) {2};
\node at (3.5,1.5) {2};
\node at (4.5,1.5) {2};
\node at (5.5,1.5) {2};
\node at (6.5,1.5) {2};

\node at (8.5,1.5) {2};
\node at (9.5,1.5) {2};
\node at (10.5,1.5) {2};
\node at (11.5,1.5) {2};

\node at (1.5,2.5) {2};
\node at (2.5,2.5) {2};
\node at (3.5,2.5) {2};
\node at (4.5,2.5) {2};
\node at (5.5,2.5) {2};
\node at (6.5,2.5) {2};

\node at (8.5,2.5) {2};
\node at (9.5,2.5) {2};
\node at (10.5,2.5) {2};

\node at (1.5,3.5) {2};
\node at (2.5,3.5) {2};
\node at (3.5,3.5) {2};
\node at (4.5,3.5) {2};
\node at (5.5,3.5) {2};
\node at (6.5,3.5) {2};

\node at (8.5,3.5) {2};
\node at (9.5,3.5) {2};

\node at (1.5,4.5) {2};
\node at (2.5,4.5) {2};
\node at (3.5,4.5) {2};
\node at (4.5,4.5) {2};
\node at (5.5,4.5) {2};
\node at (6.5,4.5) {2};

\node at (8.5,4.5) {2};

\node at (1.5,5.5) {2};
\node at (2.5,5.5) {2};
\node at (3.5,5.5) {2};
\node at (4.5,5.5) {2};
\node at (5.5,5.5) {2};
\node at (6.5,5.5) {2};

\node at (1.5,6.5) {2};
\node at (2.5,6.5) {2};
\node at (3.5,6.5) {2};
\node at (4.5,6.5) {2};
\node at (5.5,6.5) {2};
\node at (6.5,6.5) {2};

\node at (1.5,8.5) {2};
\node at (2.5,8.5) {2};
\node at (3.5,8.5) {2};
\node at (4.5,8.5) {2};

\node at (1.5,9.5) {2};
\node at (2.5,9.5) {2};
\node at (3.5,9.5) {2};

\node at (1.5,10.5) {2};
\node at (2.5,10.5) {2};

\node at (1.5,11.5) {2};

\foreach \x in {0.5,1.5,...,4.5}
\node at (\x,7.6) {$\vdots$};
\foreach \y in {0.5,1.5,...,4.5}
\node at (7.5,\y) {$\cdots$};
\node at (5.5,7.6) {$\ddots$};
\node at (6.5,7.6) {$\ddots$};
\node at (7.5,6.6) {$\ddots$};
\node at (7.5,5.6) {$\ddots$};
\end{tikzpicture}

\caption{The absolute values of the coefficients of $\LG(T(2,n))$ form a 2d log-concave sequence with no interior zeros, and a unimodal sequence.}\label{fig:FIG8}
\end{figure}

\begin{figure}[h!]
\begin{tikzpicture}[scale=0.8]

\draw[step=1cm,gray,thin] (0,0) grid (16,16);

\draw[->,thick] (0,0) -- (16.5,0) node[right] {$t_0$};

\draw[->,thick] (0,0) -- (0,16.5) node[above] {$t_1$};

\node[font=\tiny,left] at (0,0.5) {$-1$};
\node[font=\tiny,left] at (0,1.5) {$0$};
\node[font=\tiny,left] at (0,2.5) {$1$};
\node[font=\tiny,left] at (0,3.5) {$2$};
\node[font=\tiny,left] at (0,4.5) {$3$};
\node[font=\tiny,left] at (0,5.5) {$4$};
\node[font=\tiny,left] at (0,6.5) {$5$};
\node[left] at (0,7.5) {$\vdots$};
\node[left] at (0,8.5) {$\vdots$};
\node[font=\tiny,left] at (0,9.5) {$l-4$};
\node[font=\tiny,left] at (0,10.5) {$l-3$};
\node[font=\tiny,left] at (0,11.5) {$l-2$};
\node[font=\tiny,left] at (0,12.5) {$l-1$};
\node[font=\tiny,left] at (0,13.5) {$l$};
\node[font=\tiny,left] at (0,14.5) {$l+1$};

\node[font=\tiny,below] at (0.5,0) {$-1$};
\node[font=\tiny,below] at (1.5,0) {$0$};
\node[font=\tiny,below] at (2.5,0) {$1$};
\node[font=\tiny,below] at (3.5,0) {$2$};
\node[font=\tiny,below] at (4.5,0) {$3$};
\node[font=\tiny,below] at (5.5,0) {$4$};
\node[font=\tiny,below] at (6.5,0) {$5$};
\node[below] at (7.5,0) {$\cdots$};
\node[below] at (8.5,0) {$\cdots$};
\node[font=\tiny,below] at (9.5,0) {$l-4$};
\node[font=\tiny,below] at (10.5,0) {$l-3$};
\node[font=\tiny,below] at (11.5,0) {$l-2$};
\node[font=\tiny,below] at (12.5,0) {$l-1$};
\node[font=\tiny,below] at (13.5,0) {$l$};
\node[font=\tiny,below] at (14.5,0) {$l+1$};

\node[align=center] at (0.5,0.5) {$2l$};
\node[font=\tiny, align=center] at (1.5,0.5) {$4l$ \\ $- 1$};
\node[font=\tiny, align=center] at (0.5,1.5) {$4l$ \\ $- 1$};
\node[font=\tiny, align=center] at (0.5,2.5) {$2l$ \\ $- 1$};
\node[font=\tiny, align=center] at (2.5,0.5) {$2l$ \\ $- 1$};
\node[font=\tiny, align=center] at (1.5,1.5) {$10l$ \\ $- 3$};
\node[font=\tiny, align=center] at (2.5,1.5) {$8l$ \\ $- 6$};
\node[font=\tiny, align=center] at (1.5,2.5) {$8l$ \\ $- 6$};
\node[font=\tiny, align=center] at (3.5,1.5) {$2l$ \\ $- 3$};
\node[font=\tiny, align=center] at (1.5,3.5) {$2l$ \\ $- 3$};
\node[font=\tiny, align=center] at (2.5,2.5) {$12l$ \\ $- 15$};
\node[font=\tiny, align=center] at (3.5,2.5) {$8l$ \\ $- 14$};
\node[font=\tiny, align=center] at (2.5,3.5) {$8l$ \\ $- 14$};
\node[font=\tiny, align=center] at (3.5,3.5) {$12l$ \\ $- 27$};
\node[font=\tiny, align=center] at (4.5,2.5) {$2l$ \\ $- 5$};
\node[font=\tiny, align=center] at (2.5,4.5) {$2l$ \\ $- 5$};
\node[font=\tiny, align=center] at (3.5,4.5) {$8l$ \\ $- 22$};
\node[font=\tiny, align=center] at (4.5,3.5) {$8l$ \\ $- 22$};
\node[font=\tiny, align=center] at (5.5,3.5) {$2l$ \\ $- 7$};
\node[font=\tiny, align=center] at (3.5,5.5) {$2l$ \\ $- 7$};
\node[font=\tiny, align=center] at (4.5,4.5) {$12l$ \\ $- 39$};
\node[font=\tiny, align=center] at (5.5,5.5) {$12l$ \\ $- 51$};
\node[font=\tiny, align=center] at (6.5,6.5) {$12l$ \\ $- 63$};
\node[font=\tiny, align=center] at (5.5,4.5) {$8l$ \\ $- 30$};
\node[font=\tiny, align=center] at (4.5,5.5) {$8l$ \\ $- 30$};
\node[font=\tiny, align=center] at (6.5,5.5) {$8l$ \\ $- 38$};
\node[font=\tiny, align=center] at (5.5,6.5) {$8l$ \\ $- 38$};
\node[font=\tiny, align=center] at (6.5,4.5) {$2l$ \\ $- 9$};
\node[font=\tiny, align=center] at (4.5,6.5) {$2l$ \\ $- 9$};

\node at (7.5,7.6) {$\iddots$};
\node at (6.5,7.6) {$\iddots$};
\node at (5.5,7.6) {$\iddots$};
\node at (8.5,7.6) {$\iddots$};
\node at (9.5,7.6) {$\iddots$};
\node at (7.5,8.6) {$\iddots$};
\node at (6.5,8.6) {$\iddots$};
\node at (8.5,7.6) {$\iddots$};
\node at (8.5,6.6) {$\iddots$};
\node at (7.5,6.6) {$\iddots$};
\node at (7.5,5.6) {$\iddots$};
\node at (7.5,9.6) {$\iddots$};

\node at (8.5,8.6) {$\iddots$};
\node at (8.5,9.6) {$\iddots$};
\node at (9.5,8.6) {$\iddots$};
\node at (9.5,9.5) {45};
\node at (10.5,8.6) {$\iddots$};
\node at (8.5,10.6) {$\iddots$};
\node at (10.5,9.5) {26};
\node at (9.5,10.5) {26};
\node at (10.5,10.5) {33};
\node at (11.5,11.5) {21};
\node at (12.5,12.5) {10};
\node at (13.5,13.5) {2};
\node at (9.5,11.5) {5};
\node at (10.5,12.5) {3};
\node at (11.5,13.5) {1};

\node at (11.5,9.5) {5};
\node at (12.5,10.5) {3};
\node at (13.5,11.5) {1};

\node at (11.5,10.5) {18};
\node at (12.5,11.5) {10};
\node at (13.5,12.5) {3};

\node at (10.5,11.5) {18};
\node at (11.5,12.5) {10};
\node at (12.5,13.5) {3};

\end{tikzpicture}

\caption{The absolute values of the coefficients of $\LG(K_n)$ for $n \leq -1$.}\label{fig:FIG10}
\end{figure}

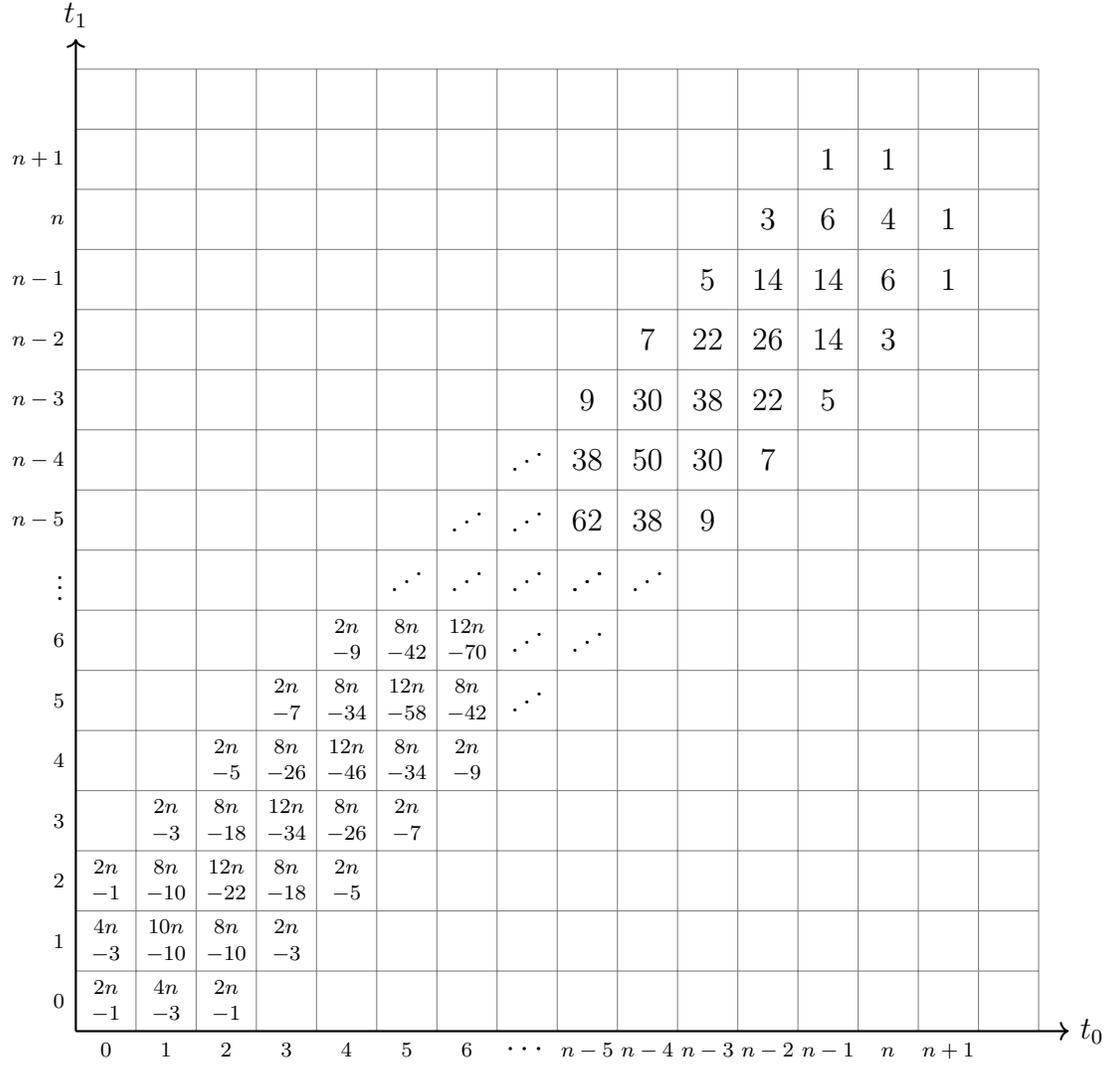
\begin{figure}[h!]
\begin{tikzpicture}[scale=0.8]

\draw[step=1cm,gray,thin] (0,0) grid (16,16);

\draw[->,thick] (0,0) -- (16.5,0) node[right] {$t_0$};

\draw[->,thick] (0,0) -- (0,16.5) node[above] {$t_1$};

\node[font=\tiny,left] at (0,0.5) {$0$};
\node[font=\tiny,left] at (0,1.5) {$1$};
\node[font=\tiny,left] at (0,2.5) {$2$};
\node[font=\tiny,left] at (0,3.5) {$3$};
\node[font=\tiny,left] at (0,4.5) {$4$};
\node[font=\tiny,left] at (0,5.5) {$5$};
\node[font=\tiny,left] at (0,6.5) {$6$};
\node[left] at (0,7.5) {$\vdots$};
\node[font=\tiny,left] at (0,8.5) {$n-5$};
\node[font=\tiny,left] at (0,9.5) {$n-4$};
\node[font=\tiny,left] at (0,10.5) {$n-3$};
\node[font=\tiny,left] at (0,11.5) {$n-2$};
\node[font=\tiny,left] at (0,12.5) {$n-1$};
\node[font=\tiny,left] at (0,13.5) {$n$};
\node[font=\tiny,left] at (0,14.5) {$n+1$};

\node[font=\tiny,below] at (0.5,0) {$0$};
\node[font=\tiny,below] at (1.5,0) {$1$};
\node[font=\tiny,below] at (2.5,0) {$2$};
\node[font=\tiny,below] at (3.5,0) {$3$};
\node[font=\tiny,below] at (4.5,0) {$4$};
\node[font=\tiny,below] at (5.5,0) {$5$};
\node[font=\tiny,below] at (6.5,0) {$6$};
\node[below] at (7.5,0) {$\cdots$};
\node[font=\tiny,below] at (8.5,0) {$n-5$};
\node[font=\tiny,below] at (9.5,0) {$n-4$};
\node[font=\tiny,below] at (10.5,0) {$n-3$};
\node[font=\tiny,below] at (11.5,0) {$n-2$};
\node[font=\tiny,below] at (12.5,0) {$n-1$};
\node[font=\tiny,below,align=center] at (13.5,0) {$n$\vphantom{1}};
\node[font=\tiny,below] at (14.5,0) {$n+1$};

\node[font=\tiny, align=center] at (0.5,0.5) {$2n$ \\ $- 1$};
\node[font=\tiny, align=center] at (1.5,0.5) {$4n$ \\ $- 3$};
\node[font=\tiny, align=center] at (0.5,1.5) {$4n$ \\ $- 3$};
\node[font=\tiny, align=center] at (0.5,2.5) {$2n$ \\ $- 1$};
\node[font=\tiny, align=center] at (2.5,0.5) {$2n$ \\ $- 1$};
\node[font=\tiny, align=center] at (1.5,1.5) {$10n$ \\ $- 10$};
\node[font=\tiny, align=center] at (2.5,1.5) {$8n$ \\ $- 10$};
\node[font=\tiny, align=center] at (1.5,2.5) {$8n$ \\ $- 10$};
\node[font=\tiny, align=center] at (3.5,1.5) {$2n$ \\ $- 3$};
\node[font=\tiny, align=center] at (1.5,3.5) {$2n$ \\ $- 3$};
\node[font=\tiny, align=center] at (2.5,2.5) {$12n$ \\ $- 22$};
\node[font=\tiny, align=center] at (3.5,2.5) {$8n$ \\ $- 18$};
\node[font=\tiny, align=center] at (2.5,3.5) {$8n$ \\ $- 18$};
\node[font=\tiny, align=center] at (3.5,3.5) {$12n$ \\ $- 34$};
\node[font=\tiny, align=center] at (4.5,2.5) {$2n$ \\ $- 5$};
\node[font=\tiny, align=center] at (2.5,4.5) {$2n$ \\ $- 5$};
\node[font=\tiny, align=center] at (3.5,4.5) {$8n$ \\ $- 26$};
\node[font=\tiny, align=center] at (4.5,3.5) {$8n$ \\ $- 26$};
\node[font=\tiny, align=center] at (5.5,3.5) {$2n$ \\ $- 7$};
\node[font=\tiny, align=center] at (3.5,5.5) {$2n$ \\ $- 7$};
\node[font=\tiny, align=center] at (4.5,4.5) {$12n$ \\ $- 46$};
\node[font=\tiny, align=center] at (5.5,5.5) {$12n$ \\ $- 58$};
\node[font=\tiny, align=center] at (6.5,6.5) {$12n$ \\ $- 70$};
\node[font=\tiny, align=center] at (5.5,4.5) {$8n$ \\ $- 34$};
\node[font=\tiny, align=center] at (4.5,5.5) {$8n$ \\ $- 34$};
\node[font=\tiny, align=center] at (6.5,5.5) {$8n$ \\ $- 42$};
\node[font=\tiny, align=center] at (5.5,6.5) {$8n$ \\ $- 42$};
\node[font=\tiny, align=center] at (6.5,4.5) {$2n$ \\ $- 9$};
\node[font=\tiny, align=center] at (4.5,6.5) {$2n$ \\ $- 9$};

\node at (7.5,7.6) {$\iddots$};
\node at (6.5,7.6) {$\iddots$};
\node at (5.5,7.6) {$\iddots$};
\node at (8.5,7.6) {$\iddots$};
\node at (9.5,7.6) {$\iddots$};
\node at (7.5,8.6) {$\iddots$};
\node at (6.5,8.6) {$\iddots$};
\node at (8.5,7.6) {$\iddots$};
\node at (8.5,6.6) {$\iddots$};
\node at (7.5,6.6) {$\iddots$};
\node at (7.5,5.6) {$\iddots$};
\node at (7.5,9.6) {$\iddots$};

\node at (8.5,8.5) {62};
\node at (8.5,9.5) {38};
\node at (9.5,8.5) {38};
\node at (9.5,9.5) {50};
\node at (10.5,8.5) {9};
\node at (8.5,10.5) {9};
\node at (10.5,9.5) {30};
\node at (9.5,10.5) {30};
\node at (10.5,10.5) {38};
\node at (11.5,11.5) {26};
\node at (12.5,12.5) {14};
\node at (13.5,13.5) {4};
\node at (9.5,11.5) {7};
\node at (10.5,12.5) {5};
\node at (11.5,13.5) {3};
\node at (12.5,14.5) {1};
\node at (11.5,9.5) {7};
\node at (12.5,10.5) {5};
\node at (13.5,11.5) {3};
\node at (14.5,12.5) {1};
\node at (11.5,10.5) {22};
\node at (12.5,11.5) {14};
\node at (13.5,12.5) {6};
\node at (14.5,13.5) {1};
\node at (10.5,11.5) {22};
\node at (11.5,12.5) {14};
\node at (12.5,13.5) {6};
\node at (13.5,14.5) {1};

\end{tikzpicture}

\caption{The absolute values of the coefficients of $\LG(K_n)(t_0^{-1}, t_1^{-1})$ for $n \geq 1$.}\label{fig:FIG11}
\end{figure}

\clearpage


\bibliographystyle{hamsalpha}
\bibliography{biblio}
\end{document}